\documentclass[onesided, 11pt, a4paper]{article}

\title{Enhanced Multi-Index Monte Carlo by means of Multiple Semi-Coarsened Multigrid for Anisotropic Diffusion Problems}
\author{Pieterjan Robbe\thanks{KU Leuven, Department of Computer Science, NUMA Section, Celestijnenlaan 200A box 2402, 3001 Leuven, Belgium (\href{mailto:pieterjan.robbe@kuleuven.be}{pieterjan.robbe@kuleuven.be}, \href{mailto:dirk.nuyens@kuleuven.be}{dirk.nuyens@kuleuven.be}, \href{mailto:stefan.vandewalle@kuleuven.be}{stefan.vandewalle@kuleuven.be}).}
\and Dirk Nuyens\footnotemark[1]
\and Stefan Vandewalle\footnotemark[1]}
\date{}

\usepackage{geometry}
\usepackage{amsmath}
\usepackage{amsthm}
\usepackage{mathtools}	
	\newtheorem{theorem}{Theorem}
\usepackage{graphicx}
\usepackage{ pgfplots}
	\pgfplotsset{compat=newest}
	\usepgfplotslibrary{external}
	\usepgfplotslibrary{statistics}
	\usepgfplotslibrary{fillbetween}
	\usepgfplotslibrary{groupplots}
	\newlength\figureheight
	\newlength\figurewidth
\usepackage{pgfplotstable}
\usepackage{tikz}
	\usetikzlibrary{external}
	\usetikzlibrary{backgrounds}
	\usetikzlibrary{matrix}
  	\usetikzlibrary{shapes.geometric}
  	\usetikzlibrary{positioning}
	\tikzexternaldisable
\usepackage{dsfont}
\usepackage{booktabs}
\usepackage{amssymb}
\usepackage{mathtools}
\usepackage{gensymb}
\usepackage{hyperref}
\hypersetup{colorlinks,linkcolor=red26,urlcolor=blue26,citecolor=green26}
\usepackage{algpseudocode}
\usepackage{algorithm}
\usepackage{blkarray}

\algnewcommand\algorithmicforeach{\textbf{for each}}
\algdef{S}[FOR]{ForEach}[1]{\algorithmicforeach\ #1\ \algorithmicdo}
	
\newcommand{\figref}[1]{Figure~\ref{#1}}

\newcommand{\slopetriangle}[5]
{
    \pgfplotsextra
    {
        \pgfkeysgetvalue{/pgfplots/xmin}{\xmin}
        \pgfkeysgetvalue{/pgfplots/xmax}{\xmax}
        \pgfkeysgetvalue{/pgfplots/ymin}{\ymin}
        \pgfkeysgetvalue{/pgfplots/ymax}{\ymax}

        \pgfmathsetmacro{\xArel}{#1}
        \pgfmathsetmacro{\yArel}{#3}
        \pgfmathsetmacro{\xBrel}{#1-#2}
        \pgfmathsetmacro{\yBrel}{\yArel}
        \pgfmathsetmacro{\xCrel}{\xBrel}

        \pgfmathsetmacro{\lnxB}{\xmin*(1-(#1-#2))+\xmax*(#1-#2)}
        \pgfmathsetmacro{\lnxA}{\xmin*(1-#1)+\xmax*#1}
        \pgfmathsetmacro{\lnyA}{\ymin*(1-#3)+\ymax*#3}
        \pgfmathsetmacro{\lnyC}{\lnyA+1.1*#4*(\lnxA-\lnxB)}
        \pgfmathsetmacro{\yCrel}{\lnyC-\ymin)/(\ymax-\ymin)}
        
        \coordinate (A) at (rel axis cs:\xArel,\yArel);
        \coordinate (B) at (rel axis cs:\xBrel,\yBrel);
        \coordinate (C) at (rel axis cs:\xCrel,\yCrel);

        \draw[#5]   (A)--node[pos=0.9,yshift=1ex,xshift=0.5ex] {\scriptsize #4}
                    (B)--
                    (C)-- 
                    cycle;
    }
}

\newcommand{\plotquantile}[4]{%
	\addplot[name path=upper,draw=none,forget plot] table[x index=0,y expr=log10(\thisrowno{#1})] {#3};%
	\addplot[name path=lower,draw=none,forget plot] table[x index=0,y expr=log10(\thisrowno{#2})] {#3};%
	\addplot[fill=#4,forget plot] fill between[of=upper and lower];%
}

\newcommand{\plotquantiles}[2]{%
	\plotquantile{3}{10}{#2}{#1!15}%
	\plotquantile{4}{9}{#2}{#1!20}%
	\plotquantile{5}{8}{#2}{#1!25}%
	\plotquantile{6}{7}{#2}{#1!30}%
}

\newcommand{\plotmedian}[4]{%
	\addplot[draw=#1,mark=#2,mark size=#3,mark options={fill=white}, rounded corners] table[x index=0,y expr=log10(\thisrowno{1})] {#4};%
}

\newcommand{\addmgresnorm}[3]{%
	\pgfplotstableread[header=false]{data/resnorms_1_2_1_4_#2_#3_MG_W_2__2___level_#1.txt}\resnormsmg
	\plotquantiles{blue26}{\resnormsmg}
	\plotmedian{blue26}{square*}{.8pt}{\resnormsmg}%
}

\newcommand{\addmsgresnorm}[3]{%
	\pgfplotstableread[header=false]{data/resnorms_1_2_1_4_#2_#3_MSG_W_2__2___level_#1.txt}\resnormsmsg
	\plotquantiles{red26}{\resnormsmsg}
	\plotmedian{red26}{*}{.9pt}{\resnormsmsg}
}

\newcommand{\plotconvergencefactor}[3]{%
	\addplot [#1, dashed, forget plot] plot [error bars/.cd, y dir = both, y explicit,error bar style={solid}] table[row sep=crcr, y error index=2]{#2};
	\addplot [#1, solid, forget plot] plot [error bars/.cd, y dir = both, y explicit,error bar style={solid}] table[row sep=crcr, y error index=2]{#3};
}

\newcommand{\plotconvergencefactors}[4]{%
	\pgfplotstableread[header=false]{data/rho_1_2_1_4_#1_#2_MG_W_#3__#3__.txt}\MG
	\pgfplotstableread[header=false]{data/rho_1_2_1_4_#1_#2_MSG_W_#3__#3__.txt}\MSG
	\plotconvergencefactor{#4}{\MG}{\MSG}\label{conv_fact_W(#3,#3)}
}

\newcommand{\convergencefactorsfor}[2]{%
	\plotconvergencefactors{#1}{#2}{2}{line1}
	\plotconvergencefactors{#1}{#2}{3}{line2}
	\plotconvergencefactors{#1}{#2}{4}{line3}
	\plotconvergencefactors{#1}{#2}{5}{line4}
	\addlegendimage{default line, dashed}\label{conv_fact_MG}
	\addlegendimage{default line}\label{conv_fact_MSG}
}

\DeclareMathSymbol{\shortminus}{\mathbin}{AMSa}{"39}

\newcommand{\drawsquare}[3]{\draw[black,fill=#3] (#1,#2)--(#1,#2+1)--(#1+1,#2+1)--(#1+1,#2)--cycle;}

\pgfplotsset{
	image axis/.style={
		width=\figurewidth,
		height=\figureheight,
		xmin = 0,
		xmax = 1,
		ymin = 0,
		ymax = 1,
		axis line style={draw=none},
		xmajorticks=false,
		ymajorticks=false
	},
	resnorm axis/.style={
		width=\figurewidth,
		height=\figureheight,
		xmin = 0,
		xmax = 50,
		xtick = {0,10,...,50},
		minor xtick={5,15,...,45},
		ymax = 1,
		ymin = -17,
		ytick = {-16,-12,...,0},
		minor ytick={-14,-10,...,-2},		
		ticklabel style={font=\scriptsize},
		major tick length=2pt,
		minor tick length=2pt,
		every tick/.style={line cap=round},
		axis on top,
		clip marker paths=true,
	},
	run time axis/.style={
		width=\figurewidth,
		height=\figureheight,
		xlabel={$\varepsilon$},
		ylabel={computation time},
		legend style={draw=none, font=\scriptsize, at={(1.03,1)}, anchor=north east, fill=none, legend cell align=left}, 
		ticklabel style={font=\small},
		major tick length={2pt},
		every tick/.style={black, line cap=round},
		axis on top,
		y label style={at={(axis description cs:-0.15,.5)},anchor=south},
		x label style={at={(axis description cs:.5,-.125)},anchor=north},
	},
	convergence factor axis/.style={
		width=\figurewidth,
		height=\figureheight,
		xmin = 0,
		xmax = 265,
		xtick = {8,16,32,64,128,256},
		minor xtick={16},
		ymax = 1,
		ymin = 0,
		ticklabel style={font=\scriptsize},
		major tick length=2pt,
		minor tick length=2pt,
		every tick/.style={black, line cap=round},
		axis on top,
	},
	fancy dense dots/.style={dash pattern=on 0pt off 3\pgflinewidth},
}

\pgfplotsset{%
	default line/.style={black, semithick, mark options={solid,fill=white}, line cap=round},
	mark1/.style={mark=square*, mark size=0.9pt},
	mark2/.style={mark=*, mark size=1pt},
	mark3/.style={mark=diamond*, mark size=1.4pt},
	mark4/.style={mark=triangle*, mark size=1.4pt},
	line1/.style={default line, color=blue26, mark1},
	line2/.style={default line, color=red26, mark2},
	line3/.style={default line, color=green26, mark3},
	line4/.style={default line, color=orange26, mark4}
}
\definecolor{blue26}{RGB}{0,117,220}
\definecolor{green26}{RGB}{43,206,72}
\definecolor{red26}{RGB}{255,0,16}
\definecolor{orange26}{RGB}{255,164,5}
\definecolor{turqoise26}{RGB}{0,153,143}
\newcommand{\bsb}{{\boldsymbol{b}}}

\newcommand{\bsr}{{\boldsymbol{r}}}

\newcommand{\bsu}{{\boldsymbol{u}}}
\newcommand{\bsv}{{\boldsymbol{v}}}

\newcommand{\bsx}{{\boldsymbol{x}}}

\newcommand{\bsL}{{\boldsymbol{L}}}

\newcommand{\bszero}{{\boldsymbol{0}}} 

\newcommand{\bstau}{{\boldsymbol{\tau}}}

\newcommand{\N}{{\mathbb{N}}} 

\DeclareSymbolFont{bbold}{U}{bbold}{m}{n}
\DeclareSymbolFontAlphabet{\mathbbold}{bbold}

\newcommand{\calP}{{\mathcal{P}}}

\newcommand{\calR}{{\mathcal{R}}}

\newcommand{\frakr}{{\mathfrak{r}}}

\newcommand{\setu}{{\mathfrak{u}}}

\newcommand{\bDelta}{\boldsymbol{\Delta}}
\newcommand{\bell}{\boldsymbol{\ell}}
\newcommand{\bkappa}{\boldsymbol{\kappa}}
\newcommand{\bL}{\boldsymbol{L}}

\newcommand{\bb}{\boldsymbol{b}}
\newcommand{\be}{\boldsymbol{e}}

\newcommand{\br}{\boldsymbol{r}}
\newcommand{\bu}{\boldsymbol{u}}
\newcommand{\bx}{\boldsymbol{x}}

\newcommand{\domain}{D}
\newcommand{\ndims}{d}
\newcommand{\samplespace}{\Omega}

\newcommand{\diffcoeffname}{a}
\newcommand{\grfname}{Z}
\newcommand{\samplename}{\omega}
\newcommand{\pointname}{\bx}
\newcommand{\solname}{u}

\newcommand{\covmodelname}{C}
\newcommand{\covdistancename}{r}
\newcommand{\lengthscale}{\lambda}
\newcommand{\corrlength}{\lengthscale}
\newcommand{\smoothness}{\nu}
\newcommand{\anisotropy}{\eta}
\newcommand{\rotation}{\theta}

\newcommand{\qoi}{Q}
\newcommand{\nbofsamples}{N}
\newcommand{\level}{\ell}
\newcommand{\maxlevel}{L}
\newcommand{\mldiff}{\Delta}
\newcommand{\idx}{{\bell}}

\newcommand{\maxindex}{\bL}
\newcommand{\midiff}{\bDelta}
\newcommand{\mlratealpha}{\alpha}
\newcommand{\mlratebeta}{\beta}
\newcommand{\mlrategamma}{\gamma}

\newcommand{\idxset}{\mathcal{I}}
\newcommand{\rotationmatrix}{R}
\newcommand{\scalematrix}{\sqrt{\Lambda}}
\newcommand{\angledeg}{\theta}
\newcommand{\grid}{G}

\newcommand{\discrmatname}{A}
\newcommand{\discrrhsname}{\bb}
\newcommand{\discrsolname}{\bu}
\newcommand{\grididxone}{p}
\newcommand{\grididxtwo}{q}
\newcommand{\grididxonestart}{1}
\newcommand{\grididxonestop}{\bar{\grididxone}}
\newcommand{\grididxtwostart}{1}
\newcommand{\grididxtwostop}{\bar{\grididxtwo}}
\newcommand{\unit}{\be}
\newcommand{\msgweight}{\bkappa}
\newcommand{\weight}{\msgweight}
\newcommand{\pmf}{\pi}
\newcommand{\tol}{\varepsilon}
\newcommand{\restrict}{\calR}
\newcommand{\prolong}{\calP}
\newcommand{\indicatorfunc}{\chi}
\newcommand{\nbofsamplessub}{\tilde{\nbofsamples}}

\newcommand{\diffcoeff}{\diffcoeffname(\pointname, \samplename)}
\newcommand{\grf}{\grfname(\pointname, \samplename)}
\newcommand{\sol}{\solname(\pointname, \samplename)}
\newcommand{\rhs}{h(\pointname)}
\newcommand{\covmodel}{\covmodelname(\pointname_1,\pointname_2)}
\newcommand{\covdistance}{\covdistancename(\pointname_1,\pointname_2)}

\newcommand{\E}[1]{\mathbb{E}\mathopen{}\left[#1\right]\mathclose{}}
\newcommand{\V}[1]{\mathbb{V}\mathopen{}\left[#1\right]\mathclose{}}

\newcommand{\estimatorfor}[1]{\mathcal{#1}}

\newcommand{\bsinfty}{\boldsymbol{\infty}}
\newcommand{\prob}[1]{\mathrm{Pr}\left[{#1}\right]}
\newcommand{\cost}[1]{\text{cost}\left({#1}\right)}
\newcommand{\order}[1]{O\mathopen{}\left({#1}\mathclose{}\right)}

\newcommand{\MIMC}{\mathcal{\qoi}_{\maxlevel,\{\nbofsamples_\idx\}}}

\newcommand{\uMIMC}{\mathcal{\qoi}_{\bsinfty,\nbofsamples}}

\newcommand{\uMIMCr}{\mathcal{\qoi}^\frakr_{\bsinfty,\nbofsamples}}

\begin{document}

\maketitle

\begin{abstract}
In many models used in engineering and science, material properties are uncertain or spatially varying. For example, in geophysics, and porous media flow in particular, the uncertain permeability of the material is modelled as a random field. These random fields can be highly anisotropic. Efficient solvers, such as the Multiple Semi-Coarsened Multigrid (MSG) method, see~\cite{mulder89, naik93, oosterlee95}, are required to compute solutions for various realisations of the uncertain material. The MSG method is an extension of the classic Multigrid method, that uses additional coarse grids that are coarsened in only a single coordinate direction. In this sense, it closely resembles the extension of Multilevel Monte Carlo (MLMC)~\cite{giles08} to Multi-Index Monte Carlo (MIMC)~\cite{haji16}. We present an unbiased MIMC method that reuses the MSG coarse solutions, similar to the work in~\cite{kumar17}. Our formulation of the estimator can be interpreted as the problem of learning the unknown distribution of the number of samples across all indices, and unifies the previous work on adaptive MIMC~\cite{robbe18} and unbiased estimation~\cite{rhee15}. We analyse the cost of this new estimator theoretically and present numerical experiments with various anisotropic random fields, where the unknown coefficients in the covariance model are considered as hyperparameters. We illustrate its robustness and superiority over unbiased MIMC without sample reuse.
\end{abstract}

\section{Introduction}\label{sec:intro}

Consider the elliptic PDE with random coefficients
\begin{equation}\label{eq:PDE}
\begin{aligned}
- \nabla \cdot ( \diffcoeff \nabla \sol ) &= \rhs \qquad &&\text{with } \pointname \in \domain, \;\samplename \in \samplespace, \text{ and}\\
\sol &= 0 \qquad &&\text{on } \partial \domain,
\end{aligned}
\end{equation}
where $\domain$ is a bounded domain in $\mathbb{R}^\ndims$, $\ndims=1,2,3$, with boundary $\partial \domain$, and $\samplespace$ is the sample space of a probability space $(\samplespace, \mathcal{F}, P)$. We assume that the diffusion coefficient $\diffcoeffname : \domain \times \samplespace \rightarrow \mathbb{R}$ is given as a lognormal random field, i.e., $\diffcoeff = \exp(\grf)$, where $\grfname$ is a zero-mean Gaussian random field with given covariance function. In this paper, we consider the so-called Mat\'ern covariance function~\cite{matern60}
\begin{equation}\label{eq:matern}
\covmodel = \frac{2^{1-\smoothness}}{\Gamma(\smoothness)}(2\sqrt{\smoothness}\;\covdistance)^{\smoothness} K_{\smoothness}(2\sqrt{\smoothness}\;\covdistance)
\end{equation}
with smoothness parameter $\nu$, and where the distance function $\covdistance$ is given by
\begin{equation}
\covdistance = \|T (\pointname_1-\pointname_2)\|_2,
\end{equation}
with $T$ a suitable transformation matrix. In what follows, we are interested in the two-dimensional case ($d=2$), and assume the linear transformation matrix $T$ is a concatenation of a scaling and a rotation, i.e.,
\begin{equation}
\begin{aligned}
T = \scalematrix\rotationmatrix
\quad
\text{with}
\quad
\scalematrix = \begin{pmatrix}
1/\sqrt{\anisotropy\lengthscale} & 0 \\
0 & 1/\sqrt{\lengthscale} \\
\end{pmatrix}
\quad
\text{and}
\quad
\rotationmatrix = \begin{pmatrix}
\cos(\angledeg) & -\sin(\angledeg) \\
\sin(\angledeg) & \cos(\angledeg) \\
\end{pmatrix},
\end{aligned}
\end{equation}
where $\lengthscale$ is the length scale, $\anisotropy$ is the anisotropic ratio and $\angledeg$ is the rotation angle. Some realisations of the random field $\grf$, for relevant combinations of the parameters in Mat\'ern covariance model~\eqref{eq:matern}, are shown in~\figref{fig:grf}.

%
%
\begin{figure}
	\hspace{-3.5em}
	\tikzexternalenable%
	\tikzsetnextfilename{grfs.tex}%
\setlength{\figurewidth}{0.4\textwidth}
\setlength{\figureheight}{0.4\textwidth}
\begin{tikzpicture}[outer sep=3pt]
	\begin{groupplot}[%
		group style={
			every plot/.append style = {image axis},
			group size=3 by 4,
			 horizontal sep=1em,
			 vertical sep=1em,
			 group name=my plots
		}
	]
	\nextgroupplot[%
	]
	\addplot[on layer=axis background] graphics[xmin=0,ymin=0,xmax=1,ymax=1] {figures/grf_1|2_1|4_1|4_0.png};
	\coordinate (t1) at (axis description cs:.5,1.1);
	\pgfplotsset{
		after end axis/.append code={%
			\node [] at (t1) {$\anisotropy = 1/4$};
		}
	}

	\nextgroupplot[%
	]
	\addplot[on layer=axis background] graphics[xmin=0,ymin=0,xmax=1,ymax=1] {figures/grf_1|2_1|4_1|8_0.png};
	\coordinate (t2) at (axis description cs:.5,1.1);
	\pgfplotsset{
		after end axis/.append code={%
			\node [] at (t2) {$\anisotropy = 1/8$};
		}
	}	

	\nextgroupplot[%
	]
	\addplot[on layer=axis background] graphics[xmin=0,ymin=0,xmax=1,ymax=1] {figures/grf_1|2_1|4_1|16_0.png};
	\coordinate (t3) at (axis description cs:.5,1.1);
	\pgfplotsset{
		after end axis/.append code={%
			\node [] at (t3) {$\anisotropy = 1/16$};
		}
	}
	\coordinate (s1) at (axis description cs:1.1,.5);
	\pgfplotsset{
		after end axis/.append code={%
			\node [rotate=-90] at (s1) {$\rotation = 0\degree$};
		}
	}

	\nextgroupplot[%
	]
	\addplot[on layer=axis background] graphics[xmin=0,ymin=0,xmax=1,ymax=1] {figures/grf_1|2_1|4_1|4_10.png};

	\nextgroupplot[%
	]
	\addplot[on layer=axis background] graphics[xmin=0,ymin=0,xmax=1,ymax=1] {figures/grf_1|2_1|4_1|8_10.png};

	\nextgroupplot[%
	]
	\addplot[on layer=axis background] graphics[xmin=0,ymin=0,xmax=1,ymax=1] {figures/grf_1|2_1|4_1|16_10.png};
	\coordinate (s2) at (axis description cs:1.1,.5);
	\pgfplotsset{
		after end axis/.append code={%
			\node [rotate=-90] at (s2) {$\rotation = 10\degree$};
		}
	}

	\nextgroupplot[%
		ylabel={\vphantom{residual norm $\log_{10}(\|\bsr_i\|_h)$}},
		ylabel style={at={(axis description cs:-.2,1.1)}}
	]
	\addplot[on layer=axis background] graphics[xmin=0,ymin=0,xmax=1,ymax=1] {figures/grf_1|2_1|4_1|4_20.png};

	\nextgroupplot[%
	]
	\addplot[on layer=axis background] graphics[xmin=0,ymin=0,xmax=1,ymax=1] {figures/grf_1|2_1|4_1|8_20.png};

	\nextgroupplot[%
	]
	\addplot[on layer=axis background] graphics[xmin=0,ymin=0,xmax=1,ymax=1] {figures/grf_1|2_1|4_1|16_20.png};
	\coordinate (s3) at (axis description cs:1.1,.5);
	\pgfplotsset{
		after end axis/.append code={%
			\node [rotate=-90] at (s3) {$\rotation = 20\degree$};
		}
	}

	\nextgroupplot[%
	]
	\addplot[on layer=axis background] graphics[xmin=0,ymin=0,xmax=1,ymax=1] {figures/grf_1|2_1|4_1|4_30.png};

	\nextgroupplot[%
	]
	\addplot[on layer=axis background] graphics[xmin=0,ymin=0,xmax=1,ymax=1] {figures/grf_1|2_1|4_1|8_30.png};

	\nextgroupplot[%
	]
	\addplot[on layer=axis background] graphics[xmin=0,ymin=0,xmax=1,ymax=1] {figures/grf_1|2_1|4_1|16_30.png};
	\coordinate (s4) at (axis description cs:1.1,.5);
	\pgfplotsset{
		after end axis/.append code={%
			\node [rotate=-90] at (s4) {$\rotation = 30\degree$};
		}
	}
	
	\end{groupplot}
	
\end{tikzpicture}%
	\tikzexternaldisable%

	\caption{\label{fig:grf}Realisations of a Gaussian random field with Mat\'ern covariance function with parameters $\corrlength =1/4$, $\smoothness =1/2$, $\anisotropy \in \{1/4, 1/8, 1/16\}$ (left to right) and $\rotation \in \{0\degree, 10\degree, 20\degree, 30\degree\}$ (top to bottom).}
\end{figure}
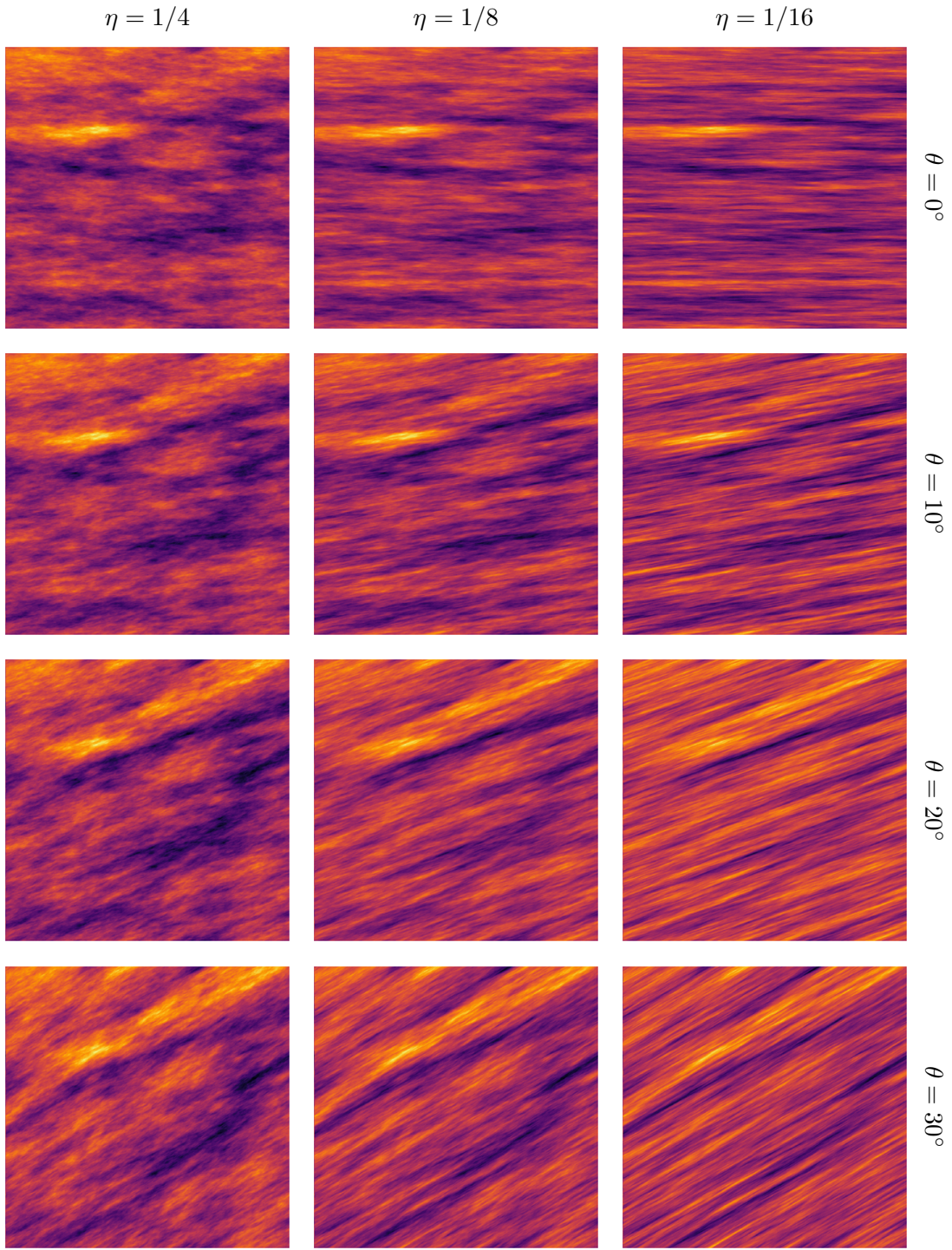

These problems find their application in groundwater hydrology, where equation~\eqref{eq:PDE} describes the steady-state single-phase flow through a porous medium, see~\cite{hiscock14}.  In this setting, $\diffcoeff$ is a random  field representing the permeability of the porous medium. In stochastic modelling of groundwater flow, it is well known that a lognormal random field may accurately represent the permeability of a naturally occurring porous medium, see, e.g.,~\cite{delhomme79}. Typical values of interest for the Mat\'ern covariance parameters in hydrology applications are a smoothness $\nu<1$, length scale $\lengthscale\ll1$ and anisotropic ratio $\anisotropy\ll1$. The solution $\sol$ of~\eqref{eq:PDE} is a random field that represents the fluid pressure at equilibrium.

Depending on the characteristics of the random field, and the way it is represented, many uncertain parameters may be required to accurately model its variability. Sampling-based methods, such as the Monte Carlo (MC) method, are the preferred tool to deal with these many uncertainties. They require the solution of a deterministic PDE for every realisation $\samplename_i \in \samplespace$, $i=1,\ldots,\nbofsamples$, with $\nbofsamples$ the number of samples. From this sample set, conclusions are drawn about the statistics of a quantity of interest $\qoi=F(\sol)$, such as the expected value, variance or higher-order moments. Here, $F$ is a functional applied to the solution $\sol$ of the PDE, for example, a point evaluation or an average over a subdomain. The classic MC method is often viewed as impractical due to the large number of expensive realisations required. It is a well-known but notorious fact that the \emph{root mean square error} decays as slowly as $1/\sqrt{N}$ for estimating an expected value, nevertheless independent of the number of uncertain parameters.

So-called \emph{multilevel methods} have been proposed to lower the MC cost, using samples of the PDE on coarser grids. These methods are known as Multilevel Monte Carlo (MLMC) methods, see, e.g.,~\cite{giles08}, or Multi-Index Monte Carlo (MIMC) methods, see~\cite{haji16}. When using a particular iterative method, called Full Multigrid (FMG), to solve the deterministic PDE underlying every sample of~\eqref{eq:PDE}, coarse solutions returned by the solver can be reused as samples in the multilevel estimator. This idea was first proposed in the context of MLMC in~\cite{kumar17}, and, in the present work, we would like to extend this concept to the MIMC setting.

The problem at hand can be decomposed into two different subproblems, that will be addressed accordingly in the remainder of this text.
\begin{enumerate}
\item[a)] We require efficient and robust deterministic multilevel solvers that compute a solution of problem~\eqref{eq:PDE} for every possible realisation $\samplename\in\samplespace$.
\item[b)] We require more efficient multilevel methods to compute statistics of a quantity of interest derived from the solution of~\eqref{eq:PDE}, that overcome the inefficiency of standard MC.
\end{enumerate}

\section{Multiple Semi-Coarsened Multigrid}\label{sec:MSG}

The convergence rate of Multigrid (MG) algorithms based on point relaxation smoothers, such as damped Jacobi or Gauss--Seidel, degenerate on problems with strong anisotropies, see~\cite{hackbusch85, trottenberg00}. Robustness can be improved by using a more powerful smoother, such as an incomplete LU (ILU)-type smoother or line ($d=2$) or plane ($d=3$) smoothers. However, these smoothers are typically more expensive, and implementation on parallel machines is nontrivial~\cite{chow06}. An alternative way to recover good multigrid convergence rates, which avoids line and plane relaxations altogether, is to use multiple coarse grids formed by semi-coarsening in each of the coordinate directions. This method is known as Multiple Semi-Coarsened Multigrid (MSG). Although first described as a nonlinear Full Approximation Scheme (FAS)-type algorithm by Mulder in~\cite{mulder89}, a linear MSG correction scheme was introduced and analysed in~\cite{naik93, oosterlee95}. See~\figref{fig:semicoarsening} for a graphical comparison of standard and semi-coarsened coarse grids. In this paper, we use MSG for robustness with respect to the random input diffusion coefficient.

%
%
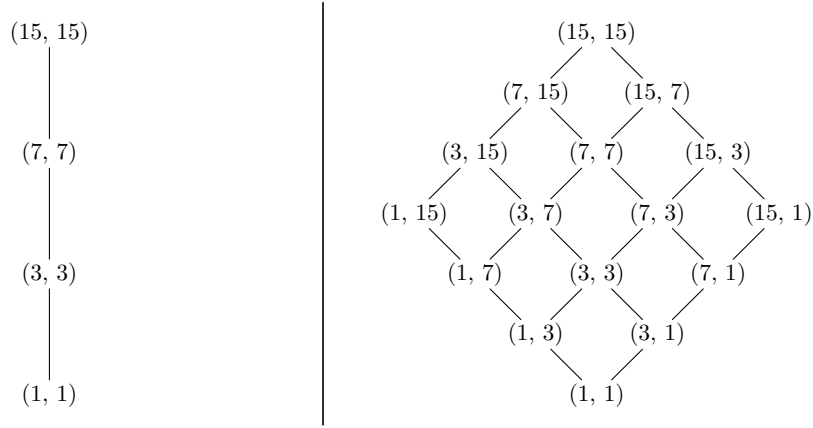
\begin{figure}
	\centering
\begin{tikzpicture}[inner sep=0pt,outer sep=0pt,scale=.8, every node/.style={transform shape}]
\foreach \x in {0,...,3}
	\foreach \y in {0,...,3}
	{
		\pgfmathtruncatemacro{\xlabel}{2*2^\x-1}
		\pgfmathtruncatemacro{\ylabel}{2*2^\y-1}
		\node [inner sep=7pt]  (\x\y) at (4.5+\x-\y,\x+\y) {};
		\node []  at (4.5+\x-\y,\x+\y) {(\xlabel, \ylabel)};
	}
\foreach \x in {0,...,3}
	\foreach \y [count=\yi] in {0,...,2}{
		\draw[-, line cap=round] (\x\y)--(\x\yi);
		\draw[-, line cap=round] (\y\x)--(\yi\x);
	}
\foreach \x in {0,...,3}{
	\pgfmathtruncatemacro{\xlabel}{2*2^\x-1}
	\node [inner sep=7pt]  (a\x) at (-4.5,2*\x) {};
	\node []  at (-4.5,2*\x) {(\xlabel, \xlabel)};
}
\foreach \x [count=\i] in {0,...,2}{
	\draw[-, line cap=round] (a\x)--(a\i);
}
\draw[-,semithick, line cap=round] (0,-0.5)--(0,6.5);
\useasboundingbox (-7,0) rectangle (7,6);
\end{tikzpicture}
	\caption{\label{fig:semicoarsening}A comparison of standard coarse grids (left) and multiple semi-coarsened grids (right) for $\ndims=2$. Number of degrees of freedom in each coordinate direction for $\bar{p} = \bar{q} = 4$.}
\end{figure}

We suppose, for simplicity, that the domain of the model problem is the unit square, i.e., $\domain=[0,1]^2$. Define a sequence of coarse grids as
\begin{equation}\label{eq:hierarchy}
\grid^{(\grididxone,\grididxtwo)} = \{ (i \cdot \delta x, \; j \cdot \delta y) : i = 0, 1, \ldots, 2^\grididxone, \; j = 0, 1, \ldots, 2^\grididxtwo\}
\end{equation}
with  $\delta x = 2^{-\grididxone}$ and $\delta y = 2^{-\grididxtwo}$, $\grididxone = \grididxonestart, \ldots, \grididxonestop$, and $\grididxtwo = \grididxtwostart, \ldots, \grididxtwostop$. The coarsest grid is thus $\grid^{(\grididxonestart, \grididxtwostart)}$, and the finest grid is $\grid^{(\grididxonestop,\grididxtwostop)}$. 
Note that a standard coarsening uses only grids with values $\grididxone=\grididxtwo$. On this hierarchy of grids, we discretise~\eqref{eq:PDE} using finite differences, resulting in a discrete version of the PDE,
\begin{equation}
\discrmatname^{(\grididxone,\grididxtwo)} \discrsolname^{(\grididxone,\grididxtwo)} = \discrrhsname^{(\grididxone,\grididxtwo)},
\end{equation}
on grid $\grid^{(\grididxone,\grididxtwo)}$. Important for the remainder of this text is that the coarse matrices $\discrmatname^{(\grididxone,\grididxtwo)}$ are obtained from direct discretisation of the differential equation, and not in variational form (as based on the \emph{Galerkin condition}).

%
%
\begin{algorithm}
\begin{algorithmic}[1]
\Statex \textbf{input:} level parameter $L$
\Statex \textbf{output:} approximations $\bsv^{(\grididxone, \grididxtwo)}$ for the solution of $A^{(\grididxone, \grididxtwo)} \bsu^{(\grididxone, \grididxtwo)} = \bsb^{(\grididxone, \grididxtwo)}, \grididxone + \grididxtwo = L$
\Statex
\Procedure{MSG-$\mu$-CYCLE}{$L$}
\If{$L=0$}
\State $\bsv^{(0, 0)} \gets \bsb^{(0, 0)} / A^{(0, 0)}$
\Else
\ForEach{grid $\grid^{(\grididxone, \grididxtwo)}$ \textbf{where} $\grididxone + \grididxtwo = L$}
	\State \textbf{call} $\text{SMOOTH}(\bsv^{(\grididxone, \grididxtwo)}, A^{(\grididxone, \grididxtwo)}, \bsb^{(\grididxone, \grididxtwo)})$ $\nu_1$ times
\EndFor
\ForEach{grid $\grid^{(\grididxone, \grididxtwo)}$ \textbf{where} $\grididxone + \grididxtwo = L$}
\State $\bsr^{(\grididxone, \grididxtwo)} \gets \bsb^{(\grididxone, \grididxtwo)} - A^{(\grididxone, \grididxtwo)}\bsv^{(\grididxone, \grididxtwo)}$ \label{chap6:alg:MSGMU:9}
\EndFor
\ForEach{grid $\grid^{(\grididxone, \grididxtwo)}$ \textbf{where} $\grididxone + \grididxtwo = L-1$}
	\State $\bsb^{(\grididxone, \grididxtwo)}\gets$ the restricted residual from~\eqref{eq:restriction}
	\State $\bsv^{(\grididxone, \grididxtwo)}\gets\boldsymbol{0}$
\EndFor
\State \textbf{call} MSG-$\mu$-CYCLE$(L-1)$ $\mu$ times
\ForEach{grid $\grid^{(\grididxone, \grididxtwo)}$ \textbf{where} $\grididxone + \grididxtwo = L$}
	\State compute the weight factors $\msgweight_x^{(\grididxone,\grididxtwo)}$ and $\msgweight_y^{(\grididxone,\grididxtwo)}$ using~\eqref{chap6:eq:weight_factors}
	\State $\bsv^{(\grididxone, \grididxtwo)}\gets\bsv^{(\grididxone, \grididxtwo)} + $ the interpolated correction from~\eqref{eq:interpolation}
	\State \textbf{call} $\text{SMOOTH}(\bsv^{(\grididxone, \grididxtwo)}, A^{(\grididxone, \grididxtwo)}, \bsb^{(\grididxone, \grididxtwo)})$ $\nu_2$ times
\EndFor
\EndIf
\EndProcedure
\end{algorithmic}
\caption{Multiple Semi-coarsened Multigrid $\mu$-cycle (recursive definition)}
\label{alg:MSGMU}
\end{algorithm}

We now assume that we have at our disposal inter-grid transfer operators between the different grids, defined in the following way:
\begin{align}
&\restrict_{x}^{(\grididxone,\grididxtwo)} : \grid^{(\grididxone+1,\grididxtwo)} \mapsto \grid^{(\grididxone, \grididxtwo)} \quad\text{and}\quad \restrict_{y}^{(\grididxone, \grididxtwo)} : \grid^{(\grididxone, \grididxtwo+1)} \mapsto \grid^{(\grididxone, \grididxtwo)}&\text{(restriction)}, \text{ and} \\
&\prolong_{x}^{(\grididxone,\grididxtwo)} : \grid^{(\grididxone-1,\grididxtwo)} \mapsto \grid^{(\grididxone, \grididxtwo)} \quad\text{and}\quad \prolong_{y}^{(\grididxone,\grididxtwo)} : \grid^{(\grididxone,\grididxtwo-1)} \mapsto \grid^{(\grididxone, \grididxtwo)}&\text{(prolongation)}. 
\end{align}

These transfer operators are simple one-dimensional operators. For example, the restriction operator can be chosen as the well-known full weighting procedure. In stencil notation, these operators are written as
\begin{align}
\restrict_{x}^{(\grididxone,\grididxtwo)} = \frac{1}{4}\scalebox{1.5}[1.5]{[}\begin{matrix}
\raisebox{1pt}{\makebox[1em]{1}} & \raisebox{1pt}{2} & \raisebox{1pt}{\makebox[1em]{1}}
\end{matrix}\scalebox{1.5}[1.5]{]}_{\grididxone+1,\grididxtwo}^{\grididxone,\grididxtwo} \quad \text{and} \quad
\restrict_{y}^{(\grididxone,\grididxtwo)} = \frac{1}{4}\begin{bmatrix}
\makebox[1em]{1} \\ 2 \\ 1
\end{bmatrix}_{\;\grididxone,\grididxtwo+1}^{\;\grididxone,\grididxtwo}.
\end{align}
Similarly, the prolongation operators $\prolong_{x}^{(\grididxone,\grididxtwo)}$ and $\prolong_{y}^{(\grididxone,\grididxtwo)}$ could be chosen as the usual linear interpolation operators.


A Multigrid cycle with $\grididxonestop+\grididxtwostop-1$ levels is now performed in the usual way, see Algorithm~\ref{alg:MSGMU}.  In each step of the algorithm, multiple semi-coarsened grids are involved (one in each coordinate direction). In the two-dimensional case considered here, essentially two semi-coarsened grids should be treated. Hence, we must specify how to transfer and combine the information from these two grids. We first describe our approach for restriction, and subsequently for prolongation. 

The residual $\bsr^{(\grididxone, \grididxtwo)} = \bsb^{(\grididxone, \grididxtwo)} - A^{(\grididxone, \grididxtwo)}\bsv^{(\grididxone, \grididxtwo)}$ is restricted to grid $\grid^{(\grididxone,\grididxtwo)}$ simply by taking the average of the restricted residuals from both finer grids $\grid^{(\grididxone+1,\grididxtwo)}$ and $\grid^{(\grididxone,\grididxtwo+1)}$, if they exist, i.e.,
\begin{equation}\label{eq:restriction}
\begin{cases}
\restrict_{x}^{(\grididxone,\grididxtwo)}\br^{(\grididxone+1,\grididxtwo)} &\text{if } \grididxtwo=\grididxtwostop,\\
\restrict_{y}^{(\grididxone,\grididxtwo)}\br^{(\grididxone,\grididxtwo+1)} &\text{if } \grididxone=\grididxonestop,\\
\frac{1}{2}\restrict_{x}^{(\grididxone,\grididxtwo)}\br^{(\grididxone+1,\grididxtwo)} + \frac{1}{2}\restrict_{y}^{(\grididxone,\grididxtwo)}\br^{(\grididxone,\grididxtwo+1)} &\text{otherwise}.\\
\end{cases}
\end{equation}

Recall that only semi-coarsening in the direction of strongest coupling can overcome the ineffectiveness of a smoother based on point relaxation \cite{hackbusch85}. In the case of grid-aligned anisotropies, only half of the MSG grids will be effective in reducing the high-frequency components of the error. Hence, it is crucial that a weighted average of interpolated corrections from both coarser grids $\grid^{(\grididxone-1,\grididxtwo)}$ and $\grid^{(\grididxone,\grididxtwo-1)}$ is used to update the solution on grid $\grid^{(\grididxone,\grididxtwo)}$, if both exist, i.e.,
\begin{equation}\label{eq:interpolation}
\begin{cases}
\prolong_{x}^{(\grididxone,\grididxtwo)}\bsv^{(\grididxone-1,\grididxtwo)} &\text{if } \grididxtwo=\grididxtwostart,\\
\prolong_{y}^{(\grididxone,\grididxtwo)}\bsv^{(\grididxone,\grididxtwo-1)} &\text{if } \grididxone=\grididxonestart,\\
\weight_x^{(\grididxone,\grididxtwo)}\prolong_{x}^{(\grididxone,\grididxtwo)}\bsv^{(\grididxone-1,\grididxtwo)} + 
\weight_y^{(\grididxone,\grididxtwo)}\prolong_{y}^{(\grididxone,\grididxtwo)}\bsv^{(\grididxone,\grididxtwo-1)} &\text{otherwise},\\
\end{cases}
\end{equation}
where $\weight_x^{(\grididxone,\grididxtwo)}$ and $\weight_y^{(\grididxone,\grididxtwo)}$ are appropriate weight factors. The original MSG paper, \cite{mulder89}, uses weight factors $\msgweight_x^{(\grididxone,\grididxtwo)} = \msgweight_y^{(\grididxone,\grididxtwo)} = 1/2$. However, these weight factors are not effective in the case of strong alignment along one of the coordinate directions, since the appropriate grids get only half of the necessary information. For our model problem, where the diffusion coefficient is given as a random field, the optimal choice for the weight factors will vary over the domain, and matrix-dependent prolongation is required to achieve acceptable convergence rates. We follow the approach from~\cite{naik93}.

Define
\begin{equation}
f_{i,j} = \cos(i \pi) \quad\text{and}\quad g_{i,j} = \cos(j \pi) \quad \text{for} \quad 0 \leq i \leq 2^\grididxone, 0 \leq j \leq 2^\grididxtwo.
\end{equation}
These are two high-frequency Fourier modes, one oscillatory in the $x$-direction, the other oscillatory in the $y$-direction, that locally look like
\begin{equation}
f = \begin{matrix}
\hphantom{\shortminus}1 & \shortminus1 & \hphantom{\shortminus}1 & \shortminus1 \\
\hphantom{\shortminus}1 & \shortminus1 & \hphantom{\shortminus}1 & \shortminus1 \\
\hphantom{\shortminus}1 & \shortminus1 & \hphantom{\shortminus}1 & \shortminus1 \\
\hphantom{\shortminus}1 & \shortminus1 & \hphantom{\shortminus}1 & \shortminus1 \\
\end{matrix}\quad\text{and}\quad
g= \begin{matrix*}[r]
1 & 1 & 1 & 1 \\
\shortminus1 & \shortminus1 & \shortminus1 & \shortminus1 \\
1 & 1 & 1 & 1 \\
\shortminus1 & \shortminus1 & \shortminus1 & \shortminus1 \\
\end{matrix*}\quad.
\end{equation}
Appropriate weight factors can be computed by applying the discrete operator $A^{(\grididxone,\grididxtwo)}$ to $f$ and $g$, that is, compute
\begin{align}
\lambda_x = A^{(\grididxone,\grididxtwo)}f \quad \text{and} \quad \lambda_y= A^{(\grididxone,\grididxtwo)}g,
\end{align}
and define the weight factors
\begin{align}\label{chap6:eq:weight_factors}
\msgweight_x^{(\grididxone,\grididxtwo)} \coloneqq \frac{\lambda_x^2}{\lambda_x^2 + \lambda_y^2} \quad \text{and} \quad \msgweight_y^{(\grididxone,\grididxtwo)} \coloneqq \frac{\lambda_y^2}{\lambda_x^2 + \lambda_y^2}.
\end{align}

It can be shown that the convergence rate of the MSG method, using these weight factors, can be made arbitrarily small if sufficient relaxation steps are performed, see~\cite{naik93}. Alternative forms for the weight factors $\weight_x^{(\grididxone,\grididxtwo)}$ and $\weight_y^{(\grididxone,\grididxtwo)}$ can be found in~\cite{dezeeuw90, dendy82, oosterlee95}.

The recursive procedure in Algorithm~\ref{alg:MSGMU} is started by constructing the hierarchy of grids given in~\eqref{eq:hierarchy}, setting an initial guess $\bsv^{(\bar{p},\bar{q})}=\bszero$ and calling the procedure with $L=\bar{p}+\bar{q}$. When $L=0$, there is only one unknown, and the system can be solved simply by inverting the matrix $A^{(0,0)}$, which is then simply a scalar. A Multigrid V-cycle is obtained when $\mu=1$, a W-cycle is obtained when $\mu=2$.

As in standard MG, it is possible to combine MSG with \emph{nested iteration}. In nested iteration, an initial approximation for an iterative method on a fine grid is provided by the computation and subsequent interpolation of solutions on coarser grids. A good initial guess for the solution on the fine grid means that just a few iterations will be made with a multigrid V- or W-cycle in order to converge the solution to discretization accuracy, see, e.g.,~\cite{hackbusch85}. Combined with MG, this method is called Full Multigrid (FMG), or, in the MSG context, Full Multiple Semi-Coarsened Multigrid (FMSG), see Algorithm~\ref{alg:FMSG}. Some remarks concerning this algorithm will be given next. First, on line~\ref{alg:FMSG:line6} of the algorithm, it is the right-hand side that is restricted, not the residual. However, equation~\eqref{eq:restriction} can still be used when the residuals $\bsr^{(\grididxone+1,\grididxtwo)}$ and $\bsr^{(\grididxone,\grididxtwo+1)}$ are replaced by right-hand side vectors $\bsb^{(\grididxone+1,\grididxtwo)}$ and $\bsb^{(\grididxone,\grididxtwo+1)}$, respectively. Secondly, on line~\ref{alg:FMSG:line11}, it is customary to use a higher-order interpolation scheme to get improved initial conditions for the subsequent V- or W-cycles. In our numerical experiments, we will use cubic interpolation. For example, the stencil for cubic interpolation in the $x$-direction is given by
\begin{equation}
\prolong_{x}^{(\grididxone,\grididxtwo)} = \frac{1}{16}\scalebox{1.5}[1.5]{[}\begin{matrix}
\raisebox{1pt}{\makebox[1em]{-1}} & \raisebox{1pt}{\makebox[1em]{0}} & \raisebox{1pt}{\makebox[1em]{9}} & \raisebox{1pt}{\makebox[1em]{16}} & \raisebox{1pt}{\makebox[1em]{9}} & \raisebox{1pt}{\makebox[1em]{0}} & \raisebox{1pt}{\makebox[1em]{-1}}
\end{matrix}\scalebox{1.5}[1.5]{]}_{\grididxone,\grididxtwo}^{\grididxone+1,\grididxtwo}
\end{equation}
or, at the left-most boundary
\begin{equation}
\prolong_{x}^{(\grididxone,\grididxtwo)} = \frac{1}{16}\scalebox{1.5}[1.5]{[}\begin{matrix}
\raisebox{1pt}{\makebox[1em]{10}} & \raisebox{1pt}{\makebox[1em]{16}} & \raisebox{1pt}{\makebox[1em]{9}} & \raisebox{1pt}{\makebox[1em]{0}} & \raisebox{1pt}{\makebox[1em]{-1}}
\end{matrix}\scalebox{1.5}[1.5]{]}_{\grididxone,\grididxtwo}^{\grididxone+1,\grididxtwo}
\end{equation}
and similar for the right-most boundary. Thirdly and lastly, in the FMG algorithm, it is common to use an adaptive number of V- or W-cycles $\nu_0$. This number is then increased until $\|\bsr^{(p,q)}\|/\|\bsb^{(p,q)}\|\leq\varepsilon$, where $\varepsilon$ is a tolerance provided by the user.

%
%
\begin{algorithm}
\begin{algorithmic}[1]
\Statex \textbf{input:} level parameter ${L}$
\Statex \textbf{output:} approximations $\bsv^{(\grididxone, \grididxtwo)}$ for the solution of $A^{(\grididxone, \grididxtwo)}\bsu^{(\grididxone, \grididxtwo)} = \bsb^{(\grididxone, \grididxtwo)}, \grididxone + \grididxtwo = L$
\Statex
\Procedure{MSG-F-CYCLE}{${L}$}
\If{$L=0$}
\State $\bsv^{(0, 0)} \gets \boldsymbol{0}$
\Else
\ForEach{grid $\grid^{(\grididxone, \grididxtwo)}$ \textbf{where} $\grididxone + \grididxtwo = {L}-1$}
	\State $\bsb^{(\grididxone, \grididxtwo)}\gets$ the restricted right-hand side from~\eqref{eq:restriction}\label{alg:FMSG:line6}
\EndFor
\State \textbf{call} \text{MSG-F-CYCLE}$({L}-1)$
\ForEach{grid $\grid^{(\grididxone, \grididxtwo)}$ \textbf{where} $\grididxone + \grididxtwo = {L}$}
	\State compute the weight factors $\msgweight_x^{(\grididxone,\grididxtwo)}$ and $\msgweight_y^{(\grididxone,\grididxtwo)}$  using~\eqref{chap6:eq:weight_factors}
	\State $\bsv^{(\grididxone, \grididxtwo)}\gets$ the interpolated solution from~\eqref{eq:interpolation}\label{alg:FMSG:line11}
\EndFor
\EndIf
\State \textbf{call} {MSG-$\mu$-CYCLE}$({L})$ $\nu_0$ times
\EndProcedure
\end{algorithmic}
\caption{Multiple Semi-coarsening Multigrid F-cycle (recursive definition)}
\label{alg:FMSG}
\end{algorithm}

We test our implementation of MSG on the model elliptic PDE with lognormal coefficients given in~\eqref{eq:PDE}, defined on the unit square $D=[0,1]^2$, where we assume a Mat\'ern covariance function for the random field, and put the source term $\rhs=1$. In \figref{fig:residual_norms}, we plot the evolution of the norm of the residual for classic MG and MSG over 50 multigrid W-cycles, for 100 samples of the random field. We used two pre- and postsmoothing steps with red-black Gauss--Seidel ($\nu_1=\nu_2=2$), and found numerically that a small damping factor before the coarse grid correction improves the overall convergence rate, see~\cite{oosterlee95}. W-cycles ($\mu=2$) are used because they are known to be generally more stable than V-cycles ($\mu=1$)~\cite{trottenberg00}. Notice the expected failure of the multigrid method with standard coarsening (MG) as the anisotropic ratio $\anisotropy$ decreases. The MSG method, on the other hand, offers fast convergence for a wide range of parameters. As the anisotropy is less grid-aligned (larger $\angledeg$) and for a fixed anisotropic ratio $\anisotropy$ (as shown in the middle column of~\figref{fig:residual_norms}), the convergence of classic MG is improved, a result also reported in~\cite{oosterlee95}.


%
%
\begin{figure}
	\centering
	\pgfmathtruncatemacro\levelcntr{3}
	\tikzexternalenable%
	\tikzsetnextfilename{resnorms.tex}%
\setlength{\figurewidth}{0.3\textwidth}
\setlength{\figureheight}{0.3\textwidth}
\begin{tikzpicture}[outer sep=3pt]
	\begin{groupplot}[%
		group style={
			every plot/.append style = {resnorm axis},
			group size=3 by 4,
			 horizontal sep=1em,
			 vertical sep=1em,
			 group name=my plots
		}
	]

	\nextgroupplot[%
		x tick label style={opacity=0},
	]
	\addmgresnorm{\levelcntr}{1|4}{0};\label{MG}
	\addmsgresnorm{\levelcntr}{1|4}{0};\label{MSG}
	\coordinate (t1) at (axis description cs:.5,1.1);
	\pgfplotsset{
		after end axis/.append code={%
			\node [] at (t1) {$\anisotropy = 1/4$};
		}
	}

	\nextgroupplot[%
		tick label style={opacity=0}
	]
	\addmgresnorm{\levelcntr}{1|8}{0}
	\addmsgresnorm{\levelcntr}{1|8}{0}
	\coordinate (t2) at (axis description cs:.5,1.1);
	\pgfplotsset{
		after end axis/.append code={%
			\node [] at (t2) {$\anisotropy = 1/8$};
		}
	}	

	\nextgroupplot[%
		tick label style={opacity=0}
	]
	\addmgresnorm{\levelcntr}{1|16}{0}
	\addmsgresnorm{\levelcntr}{1|16}{0}
	\coordinate (t3) at (axis description cs:.5,1.1);
	\pgfplotsset{
		after end axis/.append code={%
			\node [] at (t3) {$\anisotropy = 1/16$};
		}
	}
	\coordinate (s1) at (axis description cs:1.1,.5);
	\pgfplotsset{
		after end axis/.append code={%
			\node [rotate=-90] at (s1) {$\rotation = 0\degree$};
		}
	}

	\nextgroupplot[%
		x tick label style={opacity=0}
	]
	\addmgresnorm{\levelcntr}{1|4}{10}
	\addmsgresnorm{\levelcntr}{1|4}{10}

	\nextgroupplot[%
		tick label style={opacity=0}
	]
	\addmgresnorm{\levelcntr}{1|8}{10}
	\addmsgresnorm{\levelcntr}{1|8}{10}

	\nextgroupplot[%
		tick label style={opacity=0}
	]
	\addmgresnorm{\levelcntr}{1|16}{10}
	\addmsgresnorm{\levelcntr}{1|16}{10}
	\coordinate (s2) at (axis description cs:1.1,.5);
	\pgfplotsset{
		after end axis/.append code={%
			\node [rotate=-90] at (s2) {$\rotation = 10\degree$};
		}
	}

	\nextgroupplot[%
		x tick label style={opacity=0},
		ylabel={residual norm $\log_{10}(\|\bsr_i\|_h)$},
		ylabel style={at={(axis description cs:-.2,1.1)}}
	]
	\addmgresnorm{\levelcntr}{1|4}{20}
	\addmsgresnorm{\levelcntr}{1|4}{20}

	\nextgroupplot[%
		tick label style={opacity=0}
	]
	\addmgresnorm{\levelcntr}{1|8}{20}
	\addmsgresnorm{\levelcntr}{1|8}{20}

	\nextgroupplot[%
		tick label style={opacity=0}
	]
	\addmgresnorm{\levelcntr}{1|16}{20}
	\addmsgresnorm{\levelcntr}{1|16}{20}
	\coordinate (s3) at (axis description cs:1.1,.5);
	\pgfplotsset{
		after end axis/.append code={%
			\node [rotate=-90] at (s3) {$\rotation = 20\degree$};
		}
	}

	\nextgroupplot[%
	]
	\addmgresnorm{\levelcntr}{1|4}{30}
	\addmsgresnorm{\levelcntr}{1|4}{30}

	\nextgroupplot[%
		y tick label style={opacity=0},
		xlabel={number of iterations $i$},
		xlabel style={name=xlabel, at={(axis description cs:.5,-.2)}}
	]
	\addmgresnorm{\levelcntr}{1|8}{30}
	\addmsgresnorm{\levelcntr}{1|8}{30}

	\nextgroupplot[%
		y tick label style={opacity=0}
	]
	\addmgresnorm{\levelcntr}{1|16}{30}
	\addmsgresnorm{\levelcntr}{1|16}{30}
	\coordinate (s4) at (axis description cs:1.1,.5);
	\pgfplotsset{
		after end axis/.append code={%
			\node [rotate=-90] at (s4) {$\rotation = 30\degree$};
		}
	}
	\end{groupplot}
	
	\node[] () [below = .1em] at (xlabel) {};
	
\end{tikzpicture}%
	\tikzexternaldisable%

	\caption{\label{fig:residual_norms}$20$\%, $40$\%, \ldots, $80$\% quantiles of the norm of the residual after $50$ multigrid W$(2,2)$-cycles on a grid with $p = q = 6$ ($63 \times 63$ degrees of freedom) for Mat\'ern parameters $\corrlength =1/4$, $\smoothness =1/2$, $\anisotropy \in \{1/4, 1/8, 1/16\}$ (left to right) and $\rotation \in \{0\degree, 10\degree, 20\degree, 30\degree\}$ (top to bottom) based on $100$ samples. Median values are shown for MG (\ref{MG}) and MSG (\ref{MSG}).}
\end{figure}
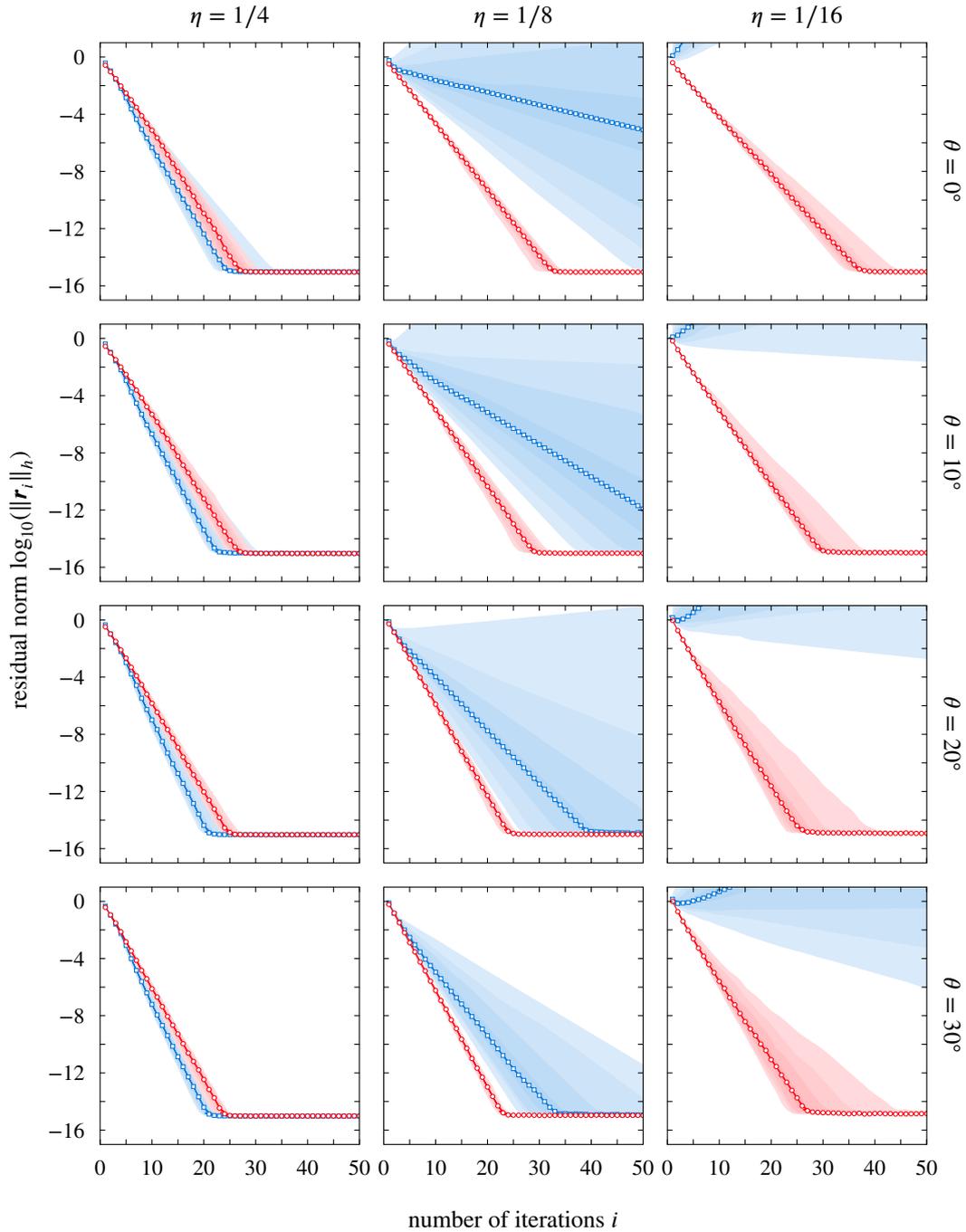

We investigate the dependence on $\angledeg$ further in~\figref{fig:convergence_factors}, where we plot the averaged convergence factor for different cycling strategies using both MG and MSG. Notice how the classic MG method requires a large number of smoothing steps to reach converge on finer grids. With $\angledeg=$ 5, no convergence is reached for grids beyond 31 $\times$ 31 unknowns, even when five pre- and postsmoothing steps are used. The MSG method, on the other hand, retains an acceptable convergence factor in all cases considered.

Numerical results for various other parameter settings can be found online at \url{https://people.cs.kuleuven.be/~pieterjan.robbe/copper2019}. Finally, we refer to~\cite{kumar18} for a Local Fourier Analysis (LFA) of the model problem with isotropic Mat\'ern covariance. 

\section{Multi-Index Monte Carlo}\label{sec:MIMC}

Recall that our aim is to compute statistics of a quantity of interest $\qoi$  derived from the solution of problem~\eqref{eq:PDE}. Let us denote the quantity of interest computed on grid $\grid^{(\grididxone,\grididxtwo)}$ by $\qoi_{\idx}$, where $\idx=(\grididxone-\grididxone_0,\grididxtwo-\grididxtwo_0)$ is a multi-index with $\grididxone_0\leqslant\grididxone\leqslant\grididxonestop$ and $\grididxtwo_0\leqslant\grididxtwo\leqslant\grididxtwostop$. Index $\idx = (0,0)$ thus corresponds to an approximation on the coarse grid $\grid^{(\grididxone_0,\grididxtwo_0)}$, that is not necessarily the coarsest grid in the hierarchy defined by~\eqref{eq:hierarchy}. A Monte Carlo estimator for the expected value of the quantity of interest $\qoi$ is just the sample average, i.e.,
\begin{equation}
\E{\qoi} \approx \estimatorfor{\qoi}_\nbofsamples \coloneqq \frac{1}{\nbofsamples} \sum_{n=1}^\nbofsamples \qoi_{\maxindex}^{(n)},
\end{equation}
where $\qoi_{\maxindex}^{(n)}$ is the $n$th sample of $\qoi$ computed on the finest grid $\grid^{(\grididxonestop,\grididxtwostop)}$.

Instead of estimating the quantity of interest directly on the finest grid, the multi-index construction, for a generic number of dimensions $\ndims$, starts from a tensor product of single-direction differences defined as
\begin{equation}
\midiff \qoi_{\idx} \coloneqq \left(\bigotimes_{i=1}^{\ndims} \mldiff_i \right) \qoi_{\idx} \quad \text{with} \quad \mldiff_i \qoi_{\idx} = \begin{cases}
\qoi_{\idx} - \qoi_{\idx-\unit_i} &\text{ if } \level_i>0\\
\qoi_{\idx} &\text{ if } \level_i=0\\
\end{cases},
\end{equation}
where $\unit_i$ is the unit vector in direction $i$. For example, with $d=2$ and $\idx=(3,4)$, we have that
\begin{align}
\midiff \qoi_{(3,4)} &= \mldiff_2 (\mldiff_1 \qoi_{(3,4)}) \nonumber\\
&= \mldiff_2 ( \qoi_{(3,4)} - \qoi_{(2,4)}) \nonumber\\
&= \qoi_{(3,4)} - \qoi_{(2,4)} - \qoi_{(3,3)} + \qoi_{(2,3)}. \nonumber
\end{align}
In general, taking a sample of $\midiff \qoi_{\idx}$ requires the solution of a PDE on $2^\ndims$ different grids. However, when using the FMSG method for our two-dimensional model problem, we get  free solutions on \emph{all} coarser grids $\grid^{(\grididxone',\grididxtwo')}$ where $\grididxone'<\grididxone$ or $\grididxtwo'<\grididxtwo$. Hence, a sample of $\midiff \qoi_{\idx}$ can be computed at the same cost of computing a sample of $\qoi_{\idx}$.

In order to obtain an efficient multi-index estimator, it is crucial that the difference $\midiff \qoi_{\idx}$ is computed from a quantity of interest that is based on a discretization of the PDE with the same underlying sample of the random field $\grf$. This will ensure that the quantities that constitute the multi-index difference are strongly positively correlated, and, hence, their difference will be small. As a consequence, the variance of the difference is heavily reduced. This means that, to reach the same \emph{mean square error}, fewer samples of the difference are required compared to sampling the quantity of interest directly. Furthermore, as $\bell\rightarrow+{}\infty$ in every coordinate, we expect the quantity of interest to converge towards the true quantity of interest, i.e., $\qoi_{\bell}\rightarrow\qoi$, and thus $\midiff \qoi_{\idx}\rightarrow0$, so that the variance of the difference decreases rapidly with $\idx$. This means that fewer and fewer (more expensive) samples are needed on finer grids.

The Multi-Index Monte Carlo estimator, proposed in~\cite{haji16}, is the sum of sample averages of multi-index differences, i.e.,
\begin{equation}\label{eq:mimcold}
\MIMC = \sum_{\idx\in\idxset(\maxlevel)} \frac{1}{\nbofsamples_{\idx}} \sum_{n=1}^{\nbofsamples_{\idx}} \midiff \qoi_{\idx}^{(n)},
\end{equation}
with $\idxset(\maxlevel)\subseteq\mathbb{N}_0^d$ a suitable set of indices, where $\maxlevel$ governs the size of the set, and $\nbofsamples_{\idx}$ is the number of samples on index $\idx$. In its current form, the multi-index estimator does not allow sample reuse. When computing the variance of~\eqref{eq:mimcold}, necessary to control the root mean square error, unwanted correlations appear, that are hard to estimate reliably.

Our \emph{unbiased} Multi-Index Monte Carlo estimator, inspired by the work in~\cite{rhee15, detommaso18}, does allow for sample reuse. The estimator can be written as
\begin{equation}\label{eq:mimc}
\uMIMC = \frac{1}{\nbofsamples}\sum_{n=1}^\nbofsamples \sum_{\bszero\leqslant\idx\leqslant\maxindex^{(n)}} \frac{1}{p_\idx} \midiff \qoi_{\idx}^{(n)},
\end{equation}
where for every sample $n$, we draw a $\ndims$-variate discrete random variable $\bsL^{(n)}$ according to some non-zero probability mass function $\pmf_\bsL$, independent from the random samples of $\midiff \qoi_\idx$, and where the constants
\begin{equation}
p_\idx\coloneqq\prob{\bszero\leqslant\idx\leqslant\maxindex} = \prob{0\leqslant\level_1\leqslant\maxlevel_1 \cup 0\leqslant\level_2\leqslant\maxlevel_2 \cup \cdots \cup 0\leqslant\level_\ndims\leqslant\maxlevel_\ndims}
\end{equation}
are the probability that $\maxindex$ is at least $\idx$, component-wise. 
The argument inside the multiple sum will be nonzero up to only a finite subset of $\N_0^\ndims$. If the probability mass $\pmf_\bsL$ is decreasing with $\bsL$, component-wise, this is mimicking the decreasing multivariate sequence of sample sizes $\{N_\idx\}$ in the classic MIMC estimator from equation~\eqref{eq:mimcold}.

We can proof unbiasedness of the MIMC estimator~\eqref{eq:mimc} by noting that
\begin{align}
\E{\uMIMC} &= \E{\frac{1}{\nbofsamples} \sum_{n=1}^\nbofsamples \sum_{\idx\in\N_0^\ndims} \frac{1}{p_\idx} \midiff \qoi_\idx^{(n)} \indicatorfunc_{[\boldsymbol{0}, \bsL^{(n)}]}(\idx)} \nonumber\\
&= \sum_{\idx\in\N_0^\ndims} \frac{1}{p_\idx} \E{ \frac{1}{\nbofsamples} \sum_{n=1}^\nbofsamples \midiff \qoi_\idx^{(n)} \indicatorfunc_{[\boldsymbol{0}, \bsL^{(n)}]}(\idx)} \nonumber\\
&= \sum_{\idx\in\N_0^\ndims} \frac{1}{p_\idx} \E{ \midiff \qoi_\idx } \E{\indicatorfunc_{[\boldsymbol{0}, \bsL^{(n)}]}(\idx)}\nonumber\\
&= \sum_{\idx\in\N_0^\ndims} \E{ \midiff \qoi_\idx } = \E{\qoi}.
\end{align}
Here, we used the multivariate indicator function, defined as the tensor product of one-dimensional indicator functions. That is, $\indicatorfunc_{[\boldsymbol{0}, \bsL^{(n)}]}(\idx) = \indicatorfunc_{[0, L_1^{(n)}]}(\level_1) \cdot \ldots \cdot \indicatorfunc_{[0, L_\ndims^{(n)}]}(\level_\ndims)$. In the last step, we used the telescoping property of the multi-index differences.

%
%
\begin{figure}
	\centering
	\tikzexternalenable%
	\tikzsetnextfilename{convergence_factors.tex}%
\setlength{\figurewidth}{0.540\textwidth}%
\setlength{\figureheight}{0.325\textwidth}%
\tikzexternaldisable
\begin{tikzpicture}[outer sep=3pt]
	\begin{groupplot}[%
		group style={
			every plot/.append style = {convergence factor axis},
			group size=2 by 2,
			 horizontal sep=1em,
			 vertical sep=1em,
			 group name=my plots
		}
	]

	\nextgroupplot[%
		x tick label style={opacity=0},
		title={$\rotation=0\degree$},
		title style={below, at={(.5,.9)}},
	]
	\convergencefactorsfor{1_16}{0}

	\nextgroupplot[%
		tick label style={opacity=0},
		title={$\rotation=10\degree$},
		title style={below, at={(.5,.9)}},
	]
	\convergencefactorsfor{1_16}{10}

	\nextgroupplot[%
		title={$\rotation=20\degree$},
		title style={below, at={(.5,.9)}},
		ylabel={convergence factor},
		ylabel style={at={(axis description cs:-.1,1.1)}}
	]
	\convergencefactorsfor{1_16}{20}

	\nextgroupplot[%
		y tick label style={opacity=0},
		title={$\rotation=30\degree$},
		title style={below, at={(.5,.9)}},
		xlabel={$\Delta x = \Delta y$},
		xlabel style={name=xlabel, at={(axis description cs:-.04,-.15)}}
	]
	\convergencefactorsfor{1_16}{30}
	\end{groupplot}
	
	\path (my plots c1r1.north west|-current bounding box.north) --
	coordinate(legendpos)
	(my plots c2r1.north east|-current bounding box.north);
	\matrix[
		matrix of nodes,
		anchor=south,
		draw,
		inner sep=2pt,
		draw
	] at ([yshift=1ex]legendpos)
	{
		\ref{conv_fact_MG} &[1em] MG\hphantom{S} &[2em] & \ref{conv_fact_W(2,2)} &[1em] W(2,2)&[2em]\ref{conv_fact_W(4,4)} &[1em] W(4,4)\\
		\ref{conv_fact_MSG} &[1em] MSG&[2em] & \ref{conv_fact_W(3,3)} &[1em] W(3,3)&[2em]\ref{conv_fact_W(5,5)} &[1em] W(5,5)\\
	};
	
\end{tikzpicture}%
%
	\tikzexternaldisable%

	\caption{\label{fig:convergence_factors}Expected convergence factors for MG (dashed line) and MSG (solid line) using different cycling strategies for Mat\'ern parameters $\corrlength =1/4$, $\smoothness =1/2$, $\anisotropy = 1/16$ and $\rotation \in \{0\degree, 10\degree, 20\degree, 30\degree\}$ (left to right and top to bottom) based on $100$ samples. Missing values indicate that some of the samples did not convergence.}
\end{figure}
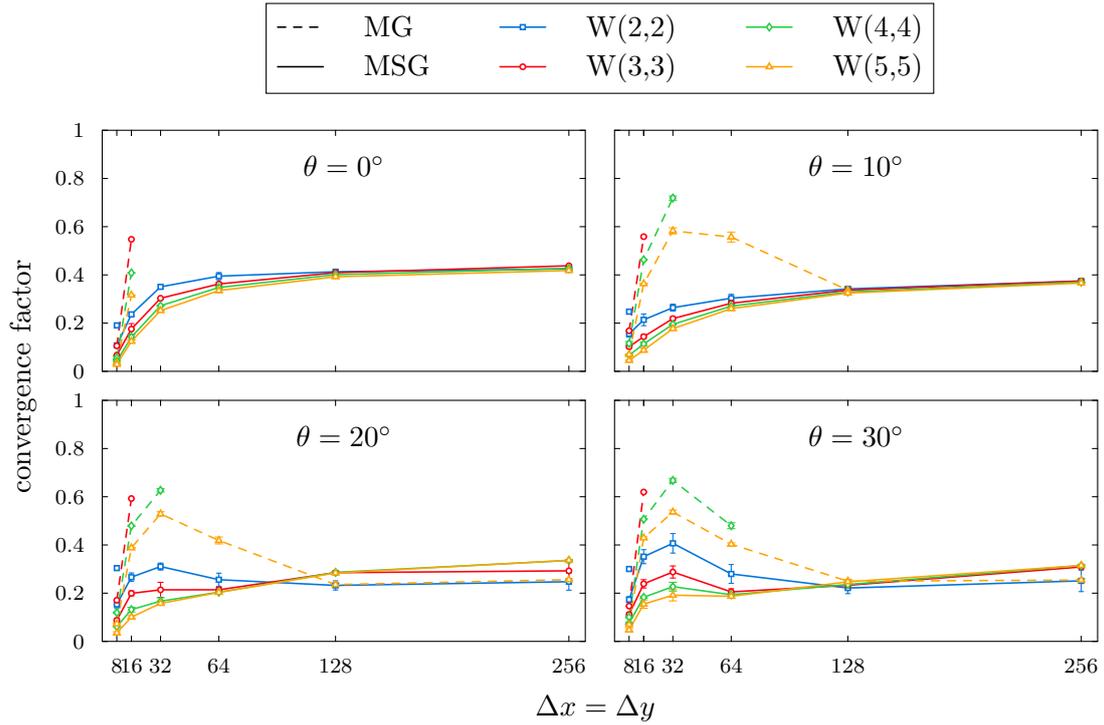

The randomisation of the index $\maxindex$ is the key difference with the classic multi-index estimator. The finite index set $\idxset(\maxlevel)$ in~\eqref{eq:mimcold} introduces an additional bias term in the expression for the root mean square error of the estimator, that needs to be controlled. Although adaptive approaches have been proposed to tackle this problem, see, e.g.,~\cite{robbe18}, obtaining a stable and accurate bound for the bias term is nontrivial. The formulation in~\eqref{eq:mimc} avoids this problem altogether, and the root mean square error of the estimator consists solely of a statistical error term, i.e.,
\begin{equation}\label{eq:rmse}
\text{RMSE}(\uMIMC) \coloneqq \sqrt{\E{\left(\uMIMC-\E{\qoi}\right)^2}}=  \sqrt{\V{\uMIMC}}.
\end{equation}
The variance of the estimator, $\V{\uMIMC}$, can be approximated using the sample variance of
\begin{equation}\label{eq:Y}
Y^{(n)} = \sum_{\bszero\leqslant\idx\leqslant\maxindex^{(n)}} \frac{1}{p_\idx} \midiff \qoi_{\idx}^{(n)},
\end{equation}
if every $n$th sample from $\maxindex$ and $\midiff \qoi_{\idx}$ is i.i.d. In particular, we have that
\begin{align}
\E{\uMIMC} &\approx E \coloneqq \frac{1}{N} \sum_{n=1}^N Y^{(n)}, \text{and} \label{eq:E}\\
\V{\uMIMC} &\approx V \coloneqq \frac{1}{N(N-1)} \sum_{n=1}^N \left( Y^{(n)} - E\right)^2.\label{eq:V}
\end{align}

At this point, it is important to stress that a single sample of $Y^{(n)}$ can be computed from only one deterministic PDE solve when using FMSG, since coarse solutions are returned for free by the multigrid solver. We will call the MIMC estimator that reuses samples Multiple Semi-coarsened Multigrid Multi-Index Monte Carlo (MSG-MIMC) and denote it by $\uMIMCr$. This is opposed to $\uMIMC$, the MIMC estimator that does not reuse samples. A basic algorithm for MSG-MIMC simulation takes $N$ draws of $\bsL$ according to a given distribution $\pmf_{\bsL}$, and performs one deterministic PDE solve using FMSG for every $\maxindex^{(n)}$, albeit on different grids. Note that this ensures the so-called \emph{downward closed} condition of the set of all indices that are added in the estimator, see~\cite{haji16}. This condition incorporates, amongst others, that the index set does not contain gaps, and ensures the validity of the telescoping sum identity used to prove unbiasedness. Algorithm~\ref{alg:MSGMIMC} provides the key instructions to implement our multi-index estimator.

%
%
\begin{algorithm}
\begin{algorithmic}[1]
\Statex \textbf{input:} a tolerance $\tol$ for the RMSE
\Statex \textbf{output:} an approximation $E$ for the mean of the quantity of interest $\qoi$, and an error estimate $\sqrt{V}$
\Statex
\Procedure{MSG-MIMC}{$\tol$}
\State $n\gets1$
\State $\pmf_\bsL\gets\text{GEOMETRIC}^\ndims(\log(2)(1+4)/2)$\label{chap6:alg:MSGMIMC:line4}
\Repeat
		\State sample $\bsL^{(n)}$ according to the probability mass function $\pmf_\bsL$\label{chap6:alg:MSGMIMC:line8}
		\State solve the PDE using FMSG to obtain $\midiff \qoi_\bstau^{(n)}$, $\boldsymbol{0} \leq \bstau \leq \bsL^{(n)}$\label{chap6:alg:MSGMIMC:line9}
		\State compute $Y^{(n)}$ using~\eqref{eq:Y}
		\State compute ${E}$ using~\eqref{eq:E} and $V$ using~\eqref{eq:V}
		\State update the probability mass function $\pmf_\bsL$
		\State $n\gets n+1$
	\Until{$\sqrt{V}\leq\tol$}
\EndProcedure
\end{algorithmic}
\caption{Multiple Semi-coarsened Multigrid Multi-Index Monte Carlo}
\label{alg:MSGMIMC}
\end{algorithm}

It remains to be determined how to make a favourable choice of $\pmf_{\maxindex}$. We rewrite our multi-index estimator as follows:
\begin{equation}\label{eq:relation_N_p}
\uMIMCr  = \sum_{n=1}^\nbofsamples \sum_{\bszero\leqslant\idx\leqslant\maxindex^{(n)}} \frac{1}{\nbofsamples p_\idx} \midiff \qoi_{\idx}^{(n)} = \sum_{n=1\vphantom{\idx}}^\nbofsamples \sum_{\bszero\leqslant\idx\leqslant\maxindex^{(n)}} \frac{1}{\nbofsamples'_{\idx}} \midiff \qoi_{\idx}^{(n)}
\end{equation}
where $\nbofsamples'_{\idx}$ denotes a discrete density of sample sizes, similar to the decreasing sequence of sample sizes $\nbofsamples_{\idx}$ in standard MIMC, see equation~\eqref{eq:mimcold}. In the latter, it is a well-known result that the optimal number of samples on each index is
\begin{equation}\label{eq:Nell}
\nbofsamples_{\idx} \simeq \sqrt{\frac{V_{\idx}}{C_{\idx}}},
\end{equation}
where $V_{\idx} \coloneqq \V{\midiff \qoi_{\idx}}$ and $C_{\idx}$ is the cost to compute a single sample of $\midiff \qoi_{\idx}$, see, e.g.,~\cite{giles08, haji16} Hence, knowing the discrete distributions of $V_{\idx}$ and $C_{\idx}$ across all $\idx$ allows one to learn the unknown probability $p_\idx$ by simple normalisation of the $N_\idx$. In a practical implementation, one can resort to either an on-the-fly computed empirical probability mass function using the (square-root of the) ratio of $V_{\idx}$ over $C_{\idx}$, or to a proposed model for their respective distributions. In our numerical experiments later on, we will opt for the first approach, while, for the remainder of this section, we choose the latter.

In particular, and as is customary in a multi-index setting, assume the expected value, variance and cost of $\midiff \qoi_{\idx}$ follow a product structure:
\begin{alignat}{2}
&E_{\idx} \coloneqq \left|\E{\midiff \qoi_{\idx}}\right| &&\leqslant c_E \prod_{j=1}^\ndims 2^{-\mlratealpha_j\level_j},\tag{A1}\label{eq:A1}\\
&V_{\idx} = \V{\midiff \qoi_{\idx}} &&\leqslant c_V \prod_{j=1}^\ndims 2^{-\mlratebeta_j\level_j},\quad\text{and}\tag{A2}\label{eq:A2}\\
&C_{\idx} = \text{cost}(\midiff \qoi_{\idx}) &&\leqslant c_C \prod_{j=1}^\ndims 2^{\mlrategamma_j\level_j},\tag{A3}\label{eq:A3}
\end{alignat}
where $c_E$, $c_V$, $c_C$ and $\mlratealpha_j$, $\mlratebeta_j$, $\mlrategamma_j$, $j=1,\ldots,\ndims$ are positive constants independent of $\bell$, see~\cite{haji16}. Using these assumptions, the optimal distribution for the number of samples $\nbofsamples_{\idx}$ in~\eqref{eq:Nell} can be written as
\begin{equation}
\nbofsamples_{\idx} \simeq \prod_{j=1}^\ndims 2^{-\frac{1}{2}(\mlrategamma_j+\mlratebeta_j)\level_j} \simeq \prod_{j=1}^\ndims \exp(-r_j \level_j),
\end{equation}
where $r_j = \frac{1}{2}\log(2)(\mlrategamma_j+\mlratebeta_j)$, for $j=1,\ldots,\ndims$. Hence, $\pmf_{\bL}$ must be chosen as a multivariate geometric distribution, i.e., the discrete equivalent of an exponential distribution, and
\begin{equation}
p_\idx = \prod_{j=1}^{\ndims} \exp(-r_j\ell_j).
\end{equation}
This explains the choice for the initial distribution on line~\ref{chap6:alg:MSGMIMC:line4} in Algorithm~\ref{alg:MSGMIMC}, where we picked $\mlratebeta_j=4$ and $\mlrategamma_j=1$, $j=1,2,\ldots,\ndims$.

We are now in a position to formulate a theorem that provides the expected cost reduction of our new MIMC estimator with sample reuse, compared to MIMC where no samples are reused, based on assumptions~\eqref{eq:A1}--\eqref{eq:A3}.

\begin{theorem}\label{thm:cost_reduction}
Let assumptions \eqref{eq:A1}--\eqref{eq:A3} hold. The cost reduction factor of the MSG-MIMC estimator is given by
\begin{equation}\label{eq:cost_reduction}
\frac{\uMIMCr}{\uMIMC} = 1 - \sum_{\substack{\setu\subseteq\{1,\ldots,\ndims\}\\\setu\ne\varnothing}} (-1)^{|\setu|+1} \prod_{j\in\setu} 2^{-\frac{1}{2}(\mlrategamma_j+\mlratebeta_j)}
\end{equation}
where $\mlratebeta_j$ is the rate of the decrease in variance in direction $j$, see~\eqref{eq:A2}, and $\mlrategamma_j$ is the rate of the increase in cost in direction $j$, see~\eqref{eq:A3}.
\end{theorem}
\begin{proof}
Suppose the algorithm requires $\nbofsamples_\idx$ samples to be taken on index $\idx$. In the MSG-MIMC method, when recycling samples, the amount of samples that remains to be taken on index $\idx$ is given by
\begin{equation}
\nbofsamplessub_\idx \coloneqq \nbofsamples_\idx - \sum_{\substack{\setu\subseteq\{1,\ldots,\ndims\}\\\setu\ne\varnothing}} (-1)^{|\setu|+1} \nbofsamples_{\idx+\unit_{\setu}}.
\end{equation}
From~\eqref{eq:Nell}, we find that
\begin{equation}
\nbofsamples_{\idx+\unit_{\setu}} = \left(\prod_{j\in\setu} 2^{-(\gamma_j+\beta_j)/2} \right) \nbofsamples_{\idx}
\end{equation}
and hence
\begin{equation}
\nbofsamplessub_\idx = \left(1 - \sum_{\substack{\setu\subseteq\{1,\ldots,\ndims\}\\\setu\ne\varnothing}} (-1)^{|\setu|+1} \prod_{j\in\setu} 2^{-(\gamma_j+\beta_j)/2}\right) \nbofsamples_\idx.
\end{equation}
Therefore, 
\begin{align}
\cost{\uMIMCr} &= \sum_{\idx\in\N_0^\ndims} \nbofsamplessub_\idx C_\idx \\
&= \left(1 - \sum_{\substack{\setu\subseteq\{1,\ldots,\ndims\}\\\setu\ne\varnothing}} (-1)^{|\setu|+1} \prod_{j\in\setu} 2^{-(\gamma_j+\beta_j)/2}\right) \sum_{\level=0}^\infty \nbofsamples_\idx C_\idx \\
&= \left(1 - \sum_{\substack{\setu\subseteq\{1,\ldots,\ndims\}\\\setu\ne\varnothing}} (-1)^{|\setu|+1} \prod_{j\in\setu} 2^{-(\gamma_j+\beta_j)/2}\right) \cost{\uMIMC}.
\end{align}
This proves the theorem.
\end{proof}

The sample reuse in the MSG-MIMC method yields only a constant factor in cost reduction, compared to unbiased MIMC without sample reuse. In applications, usually, the rate of increase in cost, $\mlrategamma_j$, $j=1,\ldots,\ndims$, is fixed, and the benefit of the sample reuse will be more pronounced when the rate of decrease in variance, $\mlratebeta_j$, $j=1,\ldots,\ndims$, is small, reflected in a slow decay of the number of samples as $\idx$ increases. This is in agreement with previous results obtained for the ML(Q)MC setting, see~\cite{robbe19}.

%
%
\begin{figure}
	\centering
	\tikzexternalenable%
	\tikzsetnextfilename{time.tex}%
\setlength{\figurewidth}{.455\textwidth}
\setlength{\figureheight}{.4\textwidth}
\begin{tikzpicture}[outer sep=4pt, baseline]
	\begin{groupplot}[%
		group style={
			group size=2 by 3,
			 horizontal sep=6em,
			 vertical sep=4em
		}
	]
	\nextgroupplot[
		run time axis,
		xmode={log},
		ymode={log},
		xmin=9.999e-5,
		xmax=1e-2,
		ymin=1e-2,
		ymax=1e4,
		ytick={1e-2,1e-1,1e0,1e1,1e2,1e3,1e4},
		yticklabels={$10^{-2}$,,$10^{0}$,,$10^{2}$,,$10^{4}$},
		clip marker paths,
		ticklabel style={font=\scriptsize}
	]
	\addplot[semithick, solid, mark1, color=red26, mark options={solid}, line cap=round] table[x index=0, y index=1] {data/time_method_1_qoi_1.txt};
	\label{plot:red}
	\addplot[fancy dense dots, semithick,, mark2, color=blue26, mark options={solid}, line cap=round] table[x index=0, y index=2] {data/time_method_1_qoi_1.txt};
	\label{plot:blue}
	\slopetriangle{.3}{.25}{.05}{2}{black}{2}
	\nextgroupplot[
		run time axis,
		ylabel={normalised time},
		xmode={log},
		xmin=9.999e-5,
		xmax=1e-2,
		ymin=.8,
		ymax=3,
		axis on top,
		clip marker paths,
		ticklabel style={font=\scriptsize},
		y tick label style={
        			/pgf/number format/.cd,
           		fixed,
            		fixed zerofill,
            		precision=1,
       	 		/tikz/.cd
    		},
	]
	\addplot[semithick, solid, mark1, color=red26, mark options={solid}, line cap=round] table[x index=0, y index=1] {data/normalized_time_method_1_qoi_1.txt};
	\addplot[semithick, fancy dense dots, color=blue26, mark2, line cap={round}, mark options={solid}, error bars/.cd, y dir = both, y explicit, error mark options={draw=none}, error bar style={solid}] table[x index=0, y index=2, y error index=4] {data/normalized_time_method_1_qoi_1.txt};
	\coordinate (s1) at (axis description cs:1.1,.5);
	\pgfplotsset{
		after end axis/.append code={%
			\node [rotate=-90] at (s1) {$Q_1$};
		}
	}
	%
	%
	%
	\nextgroupplot[
		run time axis,
		xmode={log},
		ymode={log},
		xmin=5e-3,
		xmax=1e-1,
		ymin=1e1,
		ymax=1e4,
		ytick={1e1,1e2,1e3,1e4},
		clip marker paths,
		ticklabel style={font=\scriptsize}
	]
	\addplot[semithick, solid, mark1, color=red26, mark options={solid}, line cap=round] table[x index=0, y index=1] {data/time_method_1_qoi_2.txt};
	\coordinate (A0) at ([yshift=-22pt]xticklabel cs:.25);
	\addplot[fancy dense dots, semithick,, mark2, color=blue26, mark options={solid}, line cap=round] table[x index=0, y index=2] {data/time_method_1_qoi_2.txt};
	\slopetriangle{.3}{.25}{.05}{2}{black}{2}
	\nextgroupplot[
		run time axis,
		ylabel={normalised time},
		xmode={log},
		xmin=5e-3,
		xmax=1e-1,
		ymin=.8,
		ymax=3,
		axis on top,
		clip marker paths,
		ticklabel style={font=\scriptsize},
		 y tick label style={
        			/pgf/number format/.cd,
           		fixed,
            		fixed zerofill,
            		precision=1,
       	 		/tikz/.cd
    		},
	]
	\addplot[semithick, solid, mark1, color=red26, mark options={solid}, line cap=round] table[x index=0, y index=1] {data/normalized_time_method_1_qoi_2.txt};
	\addplot[semithick, fancy dense dots, color=blue26, mark2, line cap={round}, mark options={solid}, error bars/.cd, y dir = both, y explicit, error mark options={draw=none}, error bar style={solid}] table[x index=0, y index=2, y error index=4] {data/normalized_time_method_1_qoi_2.txt};
	\coordinate (s2) at (axis description cs:1.1,.5);
	\pgfplotsset{
		after end axis/.append code={%
			\node [rotate=-90] at (s2) {$Q_2$};
		}
	}
	\nextgroupplot[
		run time axis,
		xmode={log},
		ymode={log},
		xmin=1e-2,
		xmax=1e-1,
		ymin=1e1,
		ymax=1e4,
		ytick={1e1,1e2,1e3,1e4},
		clip marker paths,
		ticklabel style={font=\scriptsize}
	]
	\addplot[semithick, solid, mark1, color=red26, mark options={solid}, line cap=round] table[x index=0, y index=1] {data/time_method_1_qoi_4.txt};
	\coordinate (A0) at ([yshift=-22pt]xticklabel cs:.25);
	\addplot[fancy dense dots, semithick,, mark2, color=blue26, mark options={solid}, line cap=round] table[x index=0, y index=2] {data/time_method_1_qoi_4.txt};
	\slopetriangle{.3}{.25}{.05}{2}{black}{2}
	\nextgroupplot[
		run time axis,
		ylabel={normalised time},
		xmode={log},
		xmin=1e-2,
		xmax=1e-1,
		ymin=.8,
		ymax=3,
		axis on top,
		clip marker paths,
		ticklabel style={font=\scriptsize},
		 y tick label style={
        			/pgf/number format/.cd,
           		fixed,
            		fixed zerofill,
            		precision=1,
       	 		/tikz/.cd
    		},
	]
	\addplot[semithick, solid, mark1, color=red26, mark options={solid}, line cap=round] table[x index=0, y index=1] {data/normalized_time_method_1_qoi_4.txt};
	\addplot[semithick, fancy dense dots, color=blue26, mark2, line cap={round}, mark options={solid}, error bars/.cd, y dir = both, y explicit, error mark options={draw=none}, error bar style={solid}] table[x index=0, y index=2, y error index=4] {data/normalized_time_method_1_qoi_4.txt};
	\coordinate (s3) at (axis description cs:1.1,.5);
	\pgfplotsset{
		after end axis/.append code={%
			\node [rotate=-90] at (s3) {$Q_3$};
		}
	}
	
	\end{groupplot}
	
	\node[] at (A0) {\textcolor{white}{x}};
\end{tikzpicture}%
	\tikzexternaldisable%

	\caption{\label{fig:time}Performance of the MSG-MIMC method with (\ref{plot:red}) and without (\ref{plot:blue}) recycling for different quantities of interest: point evaluation of the solution ($Q_1$, {top}), average value of the solution over a subdomain ($Q_2$, {middle}), flux through a boundary ($Q_3$, {bottom}).}
\end{figure}
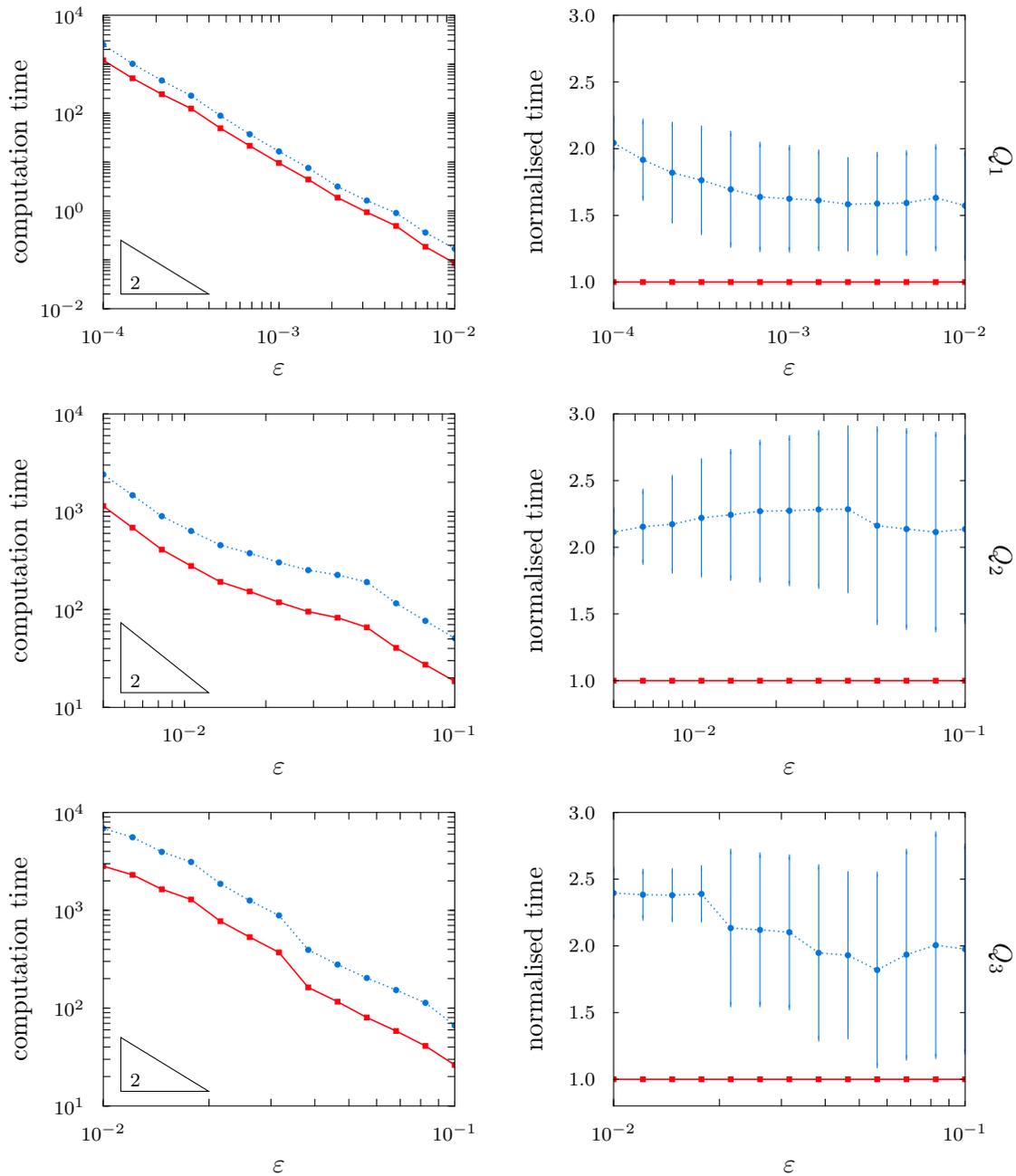

\section{Numerical results}\label{sec:numerics}

In this section, we present some numerical results for the model problem in equation~\eqref{eq:PDE}, with $D=[0,1]^2$ and $h(\bsx)=1$. We consider three quantities of interest:
\begin{enumerate}
\item A point evaluation of the solution at $\bsx=(1/2,1/2)$:
\begin{equation}
Q_1 \coloneqq u((1/2,1/2), \cdot).
\end{equation}
\item The average value of the solution over the subdomain $D_\text{sub}=[1/4,1/2]^2$:
\begin{equation}
Q_2 \coloneqq \frac{1}{\left|D_\text{sub}\right|}\int_{D_\text{sub}} u(\bsx, \cdot)\;\mathrm{d}\bsx,
\end{equation}
where the integral is approximated using a two-dimensional trapezoidal rule.
\item The flux through the rightmost side of the domain:
\begin{equation}
Q_3 \coloneqq -\int_0^1 a((1,y), \cdot) \frac{\partial u}{\partial x} ((1,y), \cdot) \mathrm{d}y,
\end{equation}
where the integral is approximated using a trapezoidal rule, and the derivative is approximated using first-order finite differences.
\end{enumerate}
In all test cases, the uncertain diffusion coefficient $a(\bsx,\omega)$ is modelled as a lognormal Gaussian random field with Mat\'ern covariance~\eqref{eq:matern} with smoothness $\smoothness=1/2$, length scale $\lengthscale=1/4$, and where the anisotropic ratio and rotation are considered as hyperparameters, uniformly distributed in $[1/16, 1/4]$ and $[-30\degree, 30\degree]$, respectively. Exact samples of the random field are computed using circulant embedding, where the field is sampled on a cube that is 
\begin{equation}
1 + 2\left\lceil\sqrt\smoothness\log_2(\max(2^{p_0}, 2^{q_0}))\right\rceil
\end{equation}
times larger in every direction, compared to the size of the physical domain, to ensure positive definiteness of the covariance matrix, see~\cite{graham17}. We should mention that, in the case of hyperparameters, the eigenvalues of the covariance matrix must be recomputed in every sample. Fortunately, they can be computed rapidly ($\order{2^{p+q}\log(2^{p+q})}$ time) using an FFT routine. See also~\cite{latz19} for an alternative method to efficiently generate samples from a Gaussian random field with hyperparameters, based on reduced basis surrogate modelling.

A coarsest mesh with $p_0=q_0=4$ was used for $Q_1$, and $p_0=q_0=16$ for $Q_2$ and $Q_3$, and the deterministic PDE is solved using FMSG with an adaptive number of W(2,2)-cycles $\nu_0$, see Section~\ref{sec:MSG}. The finest mesh on which we were able to solve the PDE consists of $2^{18}=262\,144$ degrees of freedom ($\bar{p}+\bar{q}=18$). The expected rates $\alpha_j$, $\beta_j$ and $\gamma_j$, $j=1,2$ in assumptions \eqref{eq:A1}--\eqref{eq:A3} are summarised in Table~\ref{tab:rates}, for each quantity of interest. The rate of increase in computational cost, $\gamma_j$, $j=1,2$, remains approximately constant for all quantities of interest considered. The rate of decrease in variance, $\beta_j$, $j=1,2$, diminishes, going from $Q_1$ to $Q_2$ to $Q_3$. Hence, according to Theorem~\ref{thm:cost_reduction}, we expect the largest benefit of sample reuse for $Q_3$, followed by $Q_2$, and the smallest benefit for the first quantity of interest, $Q_1$.

%
%
\begin{figure}[p]
	\centering
	\vspace{-4em}
	\tikzexternalenable%
	\tikzsetnextfilename{index_set_1.tex}%
	\input{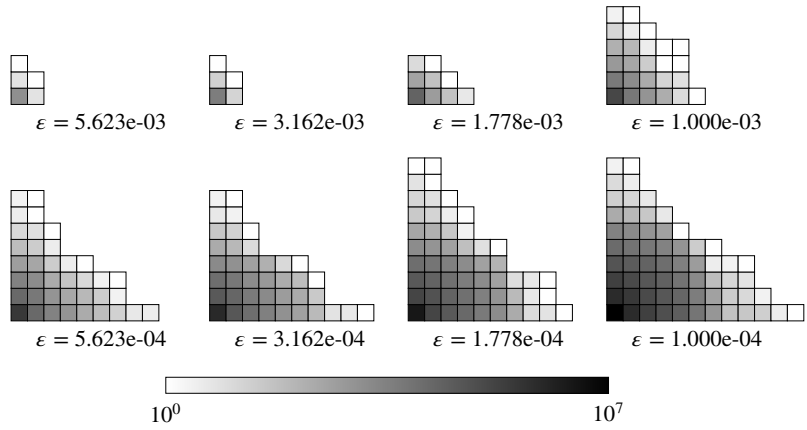}%
	\tikzexternaldisable%

	\vspace{-1em}
	\caption{\label{fig:index_set_1}Shape of the index set for the first quantity of interest $Q_1$ (point evaluation).}
\end{figure}
%
%
%

%
%
\begin{figure}[p]
	\centering
	\vspace{-2em}
	\tikzexternalenable%
	\tikzsetnextfilename{index_set_2.tex}%
	\input{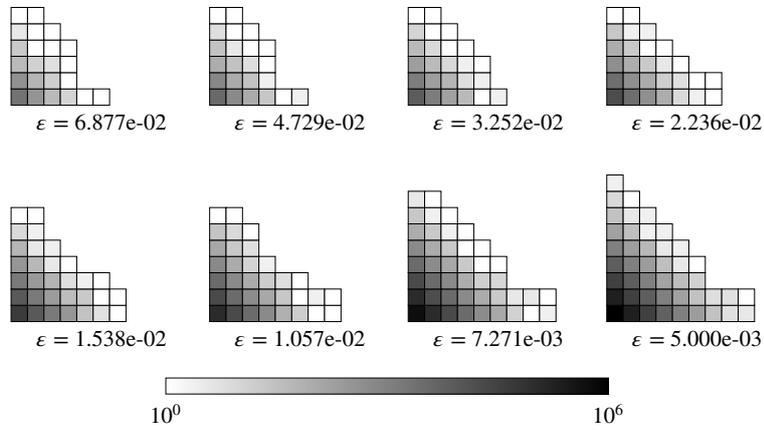}%
	\tikzexternaldisable%

	\vspace{-1em}
	\caption{\label{fig:index_set_2}Shape of the index set for the second quantity of interest $Q_2$ (average value).}
\end{figure}
%
%
%

%
%
\begin{figure}[p]
	\centering
	\tikzexternalenable%
	\tikzsetnextfilename{index_set_4.tex}%
	\input{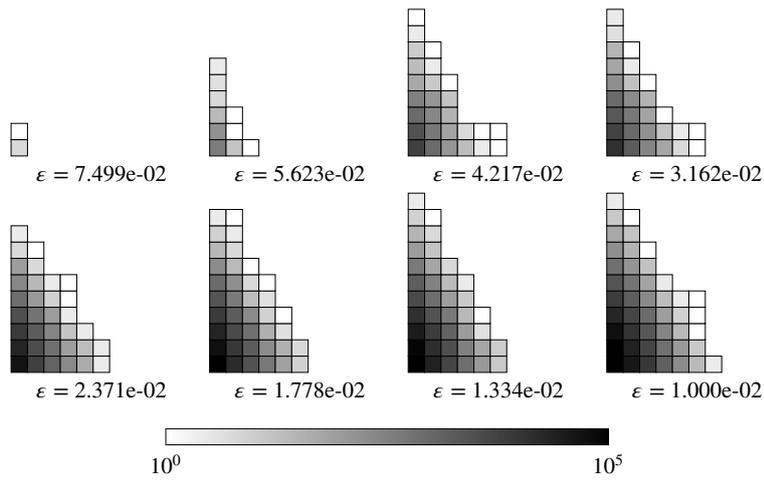}%
	\tikzexternaldisable%

	\vspace{-1em}
	\caption{\label{fig:index_set_4}Shape of the index set for the third quantity of interest $Q_3$ (flux).}
\end{figure}

\begin{table}
\centering
\begin{tabular}{cccc}\toprule
& $(\alpha_1, \alpha_2)$ & $(\beta_1, \beta_2)$ & $(\gamma_1, \gamma_2)$ \\ \midrule
$Q_1$ & $(2.02,1.97)$ & $(4.73,4.47)$ & $(1.27,1.21)$ \\
$Q_2$ & $(0.43,0.30)$ & $(2.21,2.57)$ & $(1.21,1.34)$ \\
$Q_3$ & $(0.33,0.40)$ & $(1.36,1.51)$ & $(1.49,1.26)$ \\ \bottomrule
\end{tabular}
\caption{Fitted values for the rates $\alpha_j$, $\beta_j$ and $\gamma_j$, $j=1,2$ in assumptions \eqref{eq:A1}--\eqref{eq:A3} for each quantity of interest $Q_1$, $Q_2$ and $Q_3$.} 
\label{tab:rates}
\end{table}

We compare our implementation of MSG-MIMC based on Algorithm~\ref{alg:MSGMIMC} to unbiased MIMC, where no samples are reused. We did not actually perform a simulation for the latter, but just added to the total computational cost the additional cost inferred by computing the recycled samples using FMSG. All simulations are preformed on a parallel computer with 24 logical cores. The results are outlined in~\figref{fig:time}. On the left in the figure is the average run time for 5 different random number generator seeds, plotted against a decreasing sequence of tolerances on the RMSE. On the right is the corresponding normalised run time. By recycling samples, we can achieve a cost reduction factor of 2 or more. Furthermore, the results are in agreement with our theoretical analysis: the largest speedup is observed for the third quantity of interest $Q_3$, at the bottom of the figure, because the rate of decrease in variance, $\beta_j$, $j=1,2$ is smallest. The least improvement is observed for the first quantity of interest, $Q_1$, at the top of the figure, because the rate of decrease in variance, $\beta_j$, $j=1,2$ is highest.

The shape of the index set for selected tolerances $\varepsilon$ is shown in Figures~\ref{fig:index_set_1}--\ref{fig:index_set_4}. Superimposed on the figure, using logarithmic colour codes, is the number of samples taken on each of these indices. By equation~\eqref{eq:relation_N_p}, this corresponds to a view of the learnt probability mass function $\pmf_\bsL$. Notice the effect of the anisotropic field, combined with a non-symmetric  quantity of interest, $Q_3$, on the final shape of the index set in Figure~\ref{fig:index_set_4}. It is apparent that most refinement is needed in the $y$-direction.

Finally, we should mention that the MSG-MIMC estimator for $Q_3$ is no longer unbiased for $\tol<1.334$e-$2$: the algorithm requested a sample of the multi-index difference on a grid with $p=0$ and $q=19$, which is beyond the capabilities of our current hardware, due to memory constraints. We argue, however, that this bias would be unavoidable, even when using the classic MIMC estimator.

\section*{Acknowledgments}
This research was funded by project IWT/SBO EUFORIA: ``Efficient Uncertainty Quantification For Optimization in Robust design of Industrial Applications.'' (IWT-140068) of the Agency for Innovation by Science and Technology, Flanders, Belgium.

\bibliographystyle{acm}
\bibliography{references}%

\end{document}